\documentclass[11pt]{article}
\usepackage[a4paper,top=3cm,bottom=3.5cm,left=2.5cm,right=2.5cm,marginparwidth=1.75cm]{geometry}
\usepackage[T1]{fontenc}
\usepackage{graphicx}
\usepackage[title]{appendix}
\usepackage{tikz}
\usepackage{pgfplots}
\usepackage{cutwin}

\usepackage{amsmath,yhmath}
\usepackage{amsthm}
\usepackage{amssymb}
\usepackage{hyperref}
\usepackage{cleveref}
\usepackage{epstopdf}
\usepackage{comment}
\usepackage{todonotes}
\usepackage{color}
\usepackage{float}
\usepackage[sort&compress,numbers]{natbib}
\usetikzlibrary{shapes.geometric}
\usetikzlibrary{arrows,automata,quotes}
\usepackage{iftex}
\usepackage{xfrac} 
\usepackage{amsmath,yhmath}
\usepackage{amsthm}
\usepackage{amssymb}
\usepackage{mathtools}
\usepackage{nccmath}
\usepackage{hyperref}
\usepackage{caption, subcaption}
\usepackage[title]{appendix}
\usepackage{cleveref}
\usepackage{epstopdf}
\usepackage{comment}
\usepackage{todonotes}
\usepackage{color}
\usepackage{float}
\usepackage{cutwin}
\usepackage[sort&compress,numbers]{natbib}
\usepackage{makecell}
\usepackage{array}   
\usepackage{multirow}
\usepackage{graphicx}
\usepackage{caption}
\usepackage{subcaption}
\newtheorem{thm}{Theorem}[section]
\newtheorem{cor}[thm]{Corollary}
\newtheorem{lemma}[thm]{Lemma}
\newtheorem{prop}[thm]{Proposition}
\theoremstyle{definition}
\newtheorem{definition}{Definition}[section]

\newtheorem{rem}[definition]{Remark}

\title{Transmission properties of time-dependent one-dimensional metamaterials \thanks{\footnotesize
This work was supported in part by the Swiss National Science Foundation grant number
200021--200307.}}
\author{Habib Ammari\thanks{\footnotesize Department of Mathematics, ETH Z\"urich, R\"amistrasse 101, CH-8092 Z\"urich, Switzerland (habib.ammari@math.ethz.ch, jinghao.cao@sam.math.ethz.ch, liora.rueff@sam.math.ethz.ch).} \and Jinghao Cao\footnotemark[2] \and Erik Orvehed Hiltunen\thanks{\footnotesize Department of Mathematics, Yale University, New Haven, Connecticut, USA (erik.hiltunen@yale.edu).} \and Liora Rueff\footnotemark[2] }
\date{}
\begin{document}
\maketitle
\begin{abstract}
	We solve the wave equation with periodically time-modulated material parameters in a one-dimensional high-contrast resonator structure in the subwavelength regime exactly, for which we compute the subwavelength quasifrequencies numerically using Muller's method. We prove a formula in the form of an ODE using a capacitance matrix approximation. Comparison of the exact results with the approximations reveals that the method of capacitance matrix approximation is accurate and significantly more efficient. We prove various transmission properties in the aforementioned structure and illustrate them with numerical simulations. In particular, we investigate the effect of time-modulated material parameters on the formation of degenerate points, band gaps and k-gaps.
\end{abstract}
\noindent{\textbf{Mathematics Subject Classification (MSC2000):} 35J05, 35C20, 35P20, 74J20}
\vspace{0.2cm}

\section{Introduction}
Numerous papers have tackled the problem of understanding and manipulating wave propagation in two- and three-dimensional systems with subwavelength resonant structures \cite{ammari.davies.ea2021, review2,review}. Systems of subwavelength resonant structures are of particular interest due to their ability to manipulate waves at subwavelength scales in two- and three-dimensional materials \cite{lemoult,lemoult_acoustic,leroy}. Such media are made up of a background medium and highly contrasting inclusions, which we call \textit{subwavelength resonators}. The fact that these inclusions are highly contrasting leads to \textit{subwavelength resonances}, frequencies at which the resonators interact with incident waves with wavelengths of possibly larger magnitudes \cite{feppon_cheng:hal-03697696}. This kind of structure appears in various application areas. Subwavelength resonances in highly contrasted structures can be found, for instance, in elastic media \cite{Li_Minnaert,hongyu3}, in plasmonic particles \cite{ammari2017plasmonicscalar,hyeonbae,hongyu,CARMINATI20151}, Helmholtz resonators \cite{hai,lemoult} and in dielectric high-index particles \cite{ammari2023,john}. The plethora of applications of subwavelength resonances make this topic of more general scientific interest.\par 
Wave propagation through a two- or three-dimensional structure with highly contrasting resonators is modelled by a \textit{high-contrast Helmholtz problem} \cite{AMMARI202017}. It has been shown that the high material contrast within the structure is a key assumption for the existence of resonant behaviors at subwavelength scales \cite{Ammari_bubblymedia,Minnaert_water}. The way the aforementioned Helmholtz problem is approximately solved is to use single-layer potentials based on the fundamental solution of the Laplace problem \cite{mcmpp}. Specifically, single-layer potentials are used to derive the so-called \textit{capacitance matrix}, which is used to approximate the differential equations in terms of a discrete eigenvalue problem \cite{ammari.davies.ea2021}.\par
Analogously to the two- and three-dimensional cases, the wave propagation in a one-dimensional structure is modelled by a Helmholtz problem \cite{feppon_cheng:hal-03697696}. However, we note that layer potential techniques cannot be applied to the one-dimensional setting. Thus, we must derive a distinct method to approximately solve the one-dimensional problem, which has previously been done for the finite one-dimensional case in \cite{feppon_cheng:hal-03697696}. Therefore, the results obtained in higher dimensions are not bound to hold true in the one-dimensional case, which motivates this work. Here, we seek to find a capacitance matrix approximation to the subwavelength quasifrequencies for which the quasi-periodic one-dimensional problem attains a non-trivial solution; see Definition \ref{def:quasifrequency}. Using such discrete approximation we shall be able to reproduce a number of phenomena induced by time-modulated material parameters in higher dimensional structures. \par 
Many intriguing phenomena have been shown in two- and three-dimensional high-contrast metamaterials, however, not in the one-dimensional setting. The interest in the one-dimensional case has recently risen because, in contrast to higher-dimensional cases, interactions between the resonators in one-dimensional systems only imply the nearest neighbors. 
The capacitance matrix formalism used for analysing systems of subwavelength resonators in one dimension corresponds to the tight-binding
approximation for quantum systems while in three dimensions some correspondence holds only for dilute resonators \cite{francesco}. Consequently, the one-dimensional case connects the field of high-contrast metamaterials to condensed-matter theory better. \par
Relevant recent works which focus on one-dimensional subwavelength resonators are \cite{feppon_cheng:hal-03697696} and \cite{jinghao-silvio2023}. While \cite{feppon_cheng:hal-03697696} presents the mathematical theory for the case of finitely many resonators aligned in one dimension, \cite{jinghao-silvio2023} considers the existence and characterization of topologically protected edge modes arising from defects in the periodicity of a chain of subwavelength resonators. Further relevant research has been conducted for the case of one-dimensional chains of resonators contained within a three-dimensional background medium in \cite{Ammari_Davies,AMMARI202017}. Moreover, in \cite{Lin_2022} the authors considered topological photonic materials in one dimension, they look at the consequences certain topological properties have, but not at the formation of \textit{band gaps} and \textit{non-reciprocity}. \par 
This paper particularly introduces periodically time-modulated material parameters in a quasi-periodic system of resonators, which is a natural extension to already known behaviors in one-dimensional subwavelength structures. The analogous setting in higher dimensions has been well-studied in \cite{JCP_AMMARI_HILTUNEN,ammari_cao_transmprop,Ammari_nonrecip,ammari_cao_unidirect}. We aim to investigate the formation of {band gaps}, which is a regime of subwavelength frequencies with which waves are unable to propagate through the medium, and they exponentially decay instead \cite{ammari_cao_transmprop}. It has been proven in higher dimensions that the time-modulation of the material's density leads to the emergence of band gaps \cite{ammari_cao_transmprop}. On the other hand, the time-modulation of the material's bulk modulus leads to \textit{k-gaps}  \cite{ammari_cao_transmprop}, which are band gaps in the momentum variable \cite{JCP_AMMARI_HILTUNEN}. Additionally, time-modulated material parameters induce non-reciprocity of waves propagating through two- or three-dimensional materials \cite{ammari_cao_transmprop,JCP_AMMARI_HILTUNEN,Taravati,Galiffi}. This non-reciprocity can be used to replicate spin effects from quantum systems \cite{alu1,JCP_AMMARI_HILTUNEN, Ammari_nonrecip} and to show that the unidirectional guiding phenomenon is not particular to quantum systems \cite{haldane,haldane2}. The understanding of the coupling between time-modulated material parameters and the occurrence of band gaps, k-gaps and non-reciprocity is meaningful to the field of metamaterials. In this paper we aim to prove these three observances in the case of a one-dimensional periodic structure. \par 
We start by providing an overview of the problem setting and introduce the governing equations in the form of a Helmholtz equation with suitable boundary conditions in \hyperref[sec:chpt2]{Section 2}. We particularly assume quasi-periodicity of the problem and the material parameters to be periodically time-modulated, which makes a new contribution to the understanding of subwavelength resonance phenomena in one dimension. In \hyperref[sec:chpt3]{Section 3} we introduce a scheme to solve the governing equations exactly in order to find the subwavelength quasifrequencies, for which we make use of the Dirichlet-to-Neumann approach. In \hyperref[sec:chpt4]{Section 4} we provide a brief explanation of Muller's method -- the root-finding algorithm used to solve the Helmholtz equations. In \hyperref[sec:chpt5]{Section 5} we shift our attention to a further novel contribution of this paper, which consists of the introduction of a capacitance matrix approximation of the subwavelength quasifrequencies. We prove that such a discrete approximation is a suitable replacement for the numerical scheme solving the wave problem exactly. Lastly, we move on to apply the capacitance matrix approximation to investigate the formation of band gaps, k-gaps and degeneracies and analyze the reciprocity of the wave propagation in \hyperref[sec:chpt6]{Section 6}. We summarize our results in \hyperref[sec:chpt7]{Section 7}.

\section{Problem formulation and preliminary theory}\label{sec:chpt2}
\subsection{Problem formulation}
We seek to solve the one-dimensional wave equation on a domain composed of contrasting materials. In this section, we first introduce the setting which we shall consider in the remainder of this paper. Moreover, we define the material parameters to be time-dependent and assume quasi-periodic boundary conditions.\par  
We consider the case of a one-dimensional system of periodically reoccurring chains of $N$ disjoint subwavelength resonators $D_i:=(x_i^-,x_i^+)$, where $(x_i^{\pm})_{1\leq i\leq N}$ are the $2N$ boundary points of the resonators satisfying $x_i^+<x_{i+1}^-$, for any $1\leq i\leq N-1$. We denote by $(x_i^{\pm})_{i\in\mathbb{N}}$ the infinite sequence defined by $x_{i+N}^{\pm}:=x_i^{\pm}+L$, where $L\in\mathbb{R}_{>0}$ is the period of an infinite chain of resonators. Furthermore, we denote the length of the $i$-th resonator $D_i$ by $\ell_i:=x_i^+-x_i^-$, and the length of the gap between the $i$-th and the $(i+1)$-th resonator by $\ell_{i(i+1)}:=x_{i+1}^--x_i^+$. Note that we will use the convention $\ell_{N(N+1)}:=x_{N+1}^--x_N^+=L-x_N^++x_1^-$ throughout this paper. We refer to Figure \ref{fig:1D_setting} for an illustration of the hereby introduced setting.
\begin{figure}[H]
    \centering
    \includegraphics[width=0.9\linewidth]{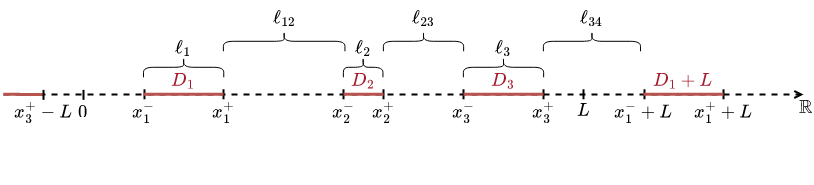}
    \caption{An illustration of the one-dimensional setting for $N=3$ resonators in the unit cell.}
    \label{fig:1D_setting}
\end{figure}\par 
In what follows, we denote by $Y:=\left(0,L\right)$ the periodic unit cell and by 
\begin{align}
    D:=\bigsqcup\limits_{i=1}^N\left(x_i^-,x_i^+\right)
\end{align}
the union of the $N$ resonators in the unit cell. With this notation, the region within $\mathbb{R}$ which is taken up by the resonators, is given by
\begin{align}
    D+L\mathbb{Z}:=\{x+kL\,:\,x\in D,\,k\in\mathbb{Z}\}.
\end{align}

\subsection{Time-dependent material parameters}
We assume that the material parameter distributions are periodic in $x$ with period $L$ and in $t$ with period $T:=2\pi/\Omega$ and are given by
\begin{align}\label{material_param}
    \kappa(x,t)=\begin{cases} 
    \kappa_0, & x \notin {D}, \\
    \kappa_{\mathrm{r}} \kappa_i(t), & x \in D_i,
    \end{cases}
    \quad \rho(x,t)= \begin{cases}\rho_0, & x \notin {D}, \\
    \rho_{\mathrm{r}} \rho_i(t), & x \in D_i.
    \end{cases}
\end{align}
Here, $\rho$ and $\kappa$ represent in acoustics the density and the bulk modulus of the material, respectively, and $\Omega$ is the frequency of the time-modulations of the material parameter distributions. \par 
We define the contrast parameter and the wave speeds by
\begin{align}
    \delta:=\frac{\rho_{\mathrm{r}}}{\rho_0},\quad v_0:=\sqrt{\frac{\kappa_0}{\rho_0}},\quad v_{\mathrm{r}}:=\sqrt{\frac{\kappa_{\mathrm{r}}}{\rho_{\mathrm{r}}}},
\end{align}
respectively. To achieve subwavelength resonance, we will assume that the contrast parameter $\delta$ is small:
\begin{equation}
\delta \ll 1.
\end{equation}
Typically, the most interesting regime of the frequency of modulations of $\rho_i(t)$ and $\kappa_i(t)$ is $\Omega=O(\delta^{1/2})$, \textit{i.e.}, of the same order as the static subwavelength resonances \cite{JCP_AMMARI_HILTUNEN}. This allows strong coupling between the time modulations and the response time of the structure. \par  
We aim at finding $\omega=O(\delta^{1/2})$  such that  the wave equation
\begin{equation}\label{eq:WaveEq}
    \left\{\begin{aligned} 
        &\left( \frac{\partial}{\partial t}\frac{1}{\kappa(x,t)}\frac{\partial}{\partial t}-\frac{\partial}{\partial x} \frac{1}{\rho(x,t)}\frac{\partial}{\partial x}\right)u(x,t)=0,\quad x\in\mathbb{R},\,t\in\mathbb{R},\\
        &u(x,t)\mathrm{e}^{-\mathrm{i}\omega t}\,\,\,\mathrm{is}\,\,T\mathrm{-periodic},
    \end{aligned}\right.
\end{equation}
has a non-trivial solution  $u(x,t)$ which is essentially supported in the \textit{low-frequency regime}.\par 
By substituting the time-harmonic wave field $u(x,t)= \Re \big( v(x,t)\mathrm{e}^{\mathrm{i}\omega t} \big)$ into the wave equation (\ref{eq:WaveEq}), we obtain
\begin{align}\label{eq:WaveEq_v}
    \left(-\mathrm{i}\omega+\frac{\partial}{\partial t}\right)  
    \frac{1}{\kappa(x,t)} \left(-\mathrm{i}\omega+\frac{\partial}{\partial t}\right)  
    v(x,t)-\frac{\partial}{\partial x}\left(\frac{1}{\rho(t,x)}\frac{\partial}{\partial x}v(x,t)\right)=0,\quad x\in\mathbb{R},\,t\in\mathbb{R}.
\end{align}\par 
Due to the assumption that $u(x,t)\mathrm{e}^{-\mathrm{i}\omega t}$ is $T$-periodic with respect to time $t$, we write the Fourier series expansion
\begin{align}\label{eq:uFourier}
    u(x,t)=\mathrm{e}^{\mathrm{i}\omega t}\sum\limits_{n=-\infty}^{\infty}v_n(x)\mathrm{e}^{\mathrm{i}n\Omega t}.
\end{align}
Note that any $L^2$-function $v_n(x)$ can be decomposed into a superposition of Bloch waves as follows:
\begin{align}
    v_n(x)=\frac{L}{\pi}\int_{-\pi/L}^{ \pi/L}\Hat{v}_n(x,\alpha)\mathrm{e}^{\mathrm{i}\alpha x}\,\mathrm{d}\alpha,
\end{align}
where $\alpha$ is the so-called \textit{momentum} and $\Hat{v}_n(x,\alpha)$ is $L$-periodic in $x$. The function $\Hat{v}_n$ is defined by
\begin{align}
    \Hat{v}_n(x,\alpha):=\sum\limits_{m=-\infty}^{\infty}v_n(x-mL)\mathrm{e}^{-\mathrm{i}\alpha(x-mL)},\quad\forall\,n\in\mathbb{Z}.
\end{align}
Thus, we can write
\begin{align}\label{eq:uFourierAlpha}
    u(x,t)=\mathrm{e}^{\mathrm{i}\omega t}\sum\limits_{n=-\infty}^{\infty}\int_{-\pi/L}^{ \pi/L}\Hat{v}_n(x,\alpha)\mathrm{e}^{\mathrm{i}\alpha x}\,\mathrm{d}\alpha\,\mathrm{e}^{\mathrm{i}n\Omega t}.
\end{align}
Inserting the expansion (\ref{eq:uFourierAlpha}) into the differential equation (\ref{eq:WaveEq_v}), we conclude that for any $n\in\mathbb{Z}$, $\Hat{v}_n$ must satisfy
\begin{equation}
    \left\{\begin{aligned} 
        &-\mathrm{i} {(\omega+n\Omega)}  \left(-\mathrm{i}\omega+\frac{\partial}{\partial t}\right)  \left(\frac{1}{\kappa(x,t)} \right) \Hat{v}_n-\left(\mathrm{i}\alpha+\frac{\partial}{\partial x}\right)\left(\frac{1}{\rho(x,t)}\left(\mathrm{i}\alpha+\frac{\partial}{\partial x}\right)\Hat{v}_n\right)=0,\\
        &x\mapsto\Hat{v}_n(x,\alpha)\,\,\mathrm{is}\,L\mathrm{-periodic},
    \end{aligned}\right.
\end{equation}
for $x\in\mathbb{R}$ and $t\in\mathbb{R}$.
\par
Recall that we have assumed the chain of $N$ resonators to be repeated periodically with period $L$. Therefore, we study the one-dimensional spectral problem in the unit cell $(0,L)$ for the quasi-periodic function $v_n(x,\alpha):=\Hat{v}_n(x,\alpha)\mathrm{e}^{\mathrm{i}\alpha x}$:
\begin{equation}\label{eq:1DL_system}
    \left\{\begin{aligned} 
        &\frac{\mathrm{d}^2}{\mathrm{d}x^2}v_n+\frac{\rho_0(\omega+n\Omega)^2}{\kappa_0}v_n=0 &\text{  in }\left(0,L\right)\backslash D,\\
        &\frac{\mathrm{d}^2}{\mathrm{d}x^2}v_{i,n}^*+\frac{\rho_{\mathrm{r}}(\omega+n\Omega)^2}{\kappa_{\mathrm{r}}}v_{i,n}^{**}=0 &\text{in }D_i,\\
        &\left.v_n\right|_{-}\left(x_i^{\pm}\right) =\left.v_n\right|_{+}\left(x_i^{\pm}\right) &\text {for all } 1 \leq i \leq N, \\ 
        &\left.\frac{\mathrm{d} v_{i,n}^*}{\mathrm{~d} x}\right|_{+}\left(x_i^{-}\right) =\left.\delta \frac{\mathrm{d} v_n}{\mathrm{d} x}\right|_{-}\left(x_i^{-}\right) &\text{for all } 1 \leq i \leq N, \\
        &\left.\frac{\mathrm{d} v_{i,n}^*}{\mathrm{~d} x}\right|_{-}\left(x_i^{+}\right)  =\left.\delta \frac{\mathrm{d} v_n}{\mathrm{d} x}\right|_{+}\left(x_i^{+}\right) &\text{for all } 1 \leq i \leq N,
    \end{aligned}\right.
\end{equation}
where we use the notation
\begin{align}
    \left.w\right|_{\pm}(x):=\lim_{s\rightarrow0,\,s>0}w(x\pm s).
\end{align}
The functions $v_{i, n}^*(x,\alpha)$ and $v_{i, n}^{* *}(x,\alpha)$ are defined in each resonator $D_i$ through the convolutions
\begin{equation}\label{def:conv_v}
    v_{i,n}^*(x,\alpha)=\sum_{m=-\infty}^{\infty} r_{i,m} v_{n-m}(x,\alpha), \quad v_{i,n}^{* *}(x,\alpha)=\frac{1}{\omega+n \Omega} \sum_{m=-\infty}^{\infty} k_{i,m}(\omega+(n-m) \Omega) v_{n-m}(x,\alpha),
\end{equation}
where $r_{i, m}$ and $k_{i, m}$ are the Fourier series coefficients of $1 / \rho_i(t)$ and $1 / \kappa_i(t)$, respectively.
 Furthermore, we define the wave number outside and inside the resonators corresponding to the $n$-th mode through
\begin{align}
    k^n:=\frac{\omega+n\Omega}{v_0},\quad k_{\mathrm{r}}^n:=\frac{\omega+n\Omega}{v_{\mathrm{r}}},
\end{align}
respectively. We assume that the time-modulations of $\rho_i$ and $\kappa_i$ have finite Fourier series in each resonator $D_i$, that is,
\begin{equation}
	\frac{1}{\rho_i(t)}=\sum_{n=-M}^M r_{i, n} \mathrm{e}^{\mathrm{i} n \Omega t}, \quad \frac{1}{\kappa_i(t)}=\sum_{n=-M}^M k_{i, n} \mathrm{e}^{\mathrm{i} n \Omega t}
\end{equation}
for some $M \in \mathbb{N}$ satisfying $M=O\left(\delta^{-\gamma / 2}\right)$, for some $\gamma\in(0,1)$ \cite{JCP_AMMARI_HILTUNEN}. Note that the solution to (\ref{eq:1DL_system}) is invariant under scaling. Hence, we can assume the solution to be normalized. As $u$ is continuously differentiable in $t$, we have 
\begin{equation}\label{eq:vn_norm}
    \Vert{v_n}\Vert_{2} = o\left(\frac{1}{n}\right) \quad \text{as}\ n\rightarrow \infty,
\end{equation}
where $\Vert \cdot \Vert_{2}$ denotes the $L^2$-norm on $(0,L)$. 
Due to folding (see Definition \ref{deffolding}), we need to specify the subwavelength quasifrequencies in terms of the oscillations in their associated modes \cite{JCP_AMMARI_HILTUNEN}. As said before, the subwavelength quasifrequencies  are those associated with Bloch modes essentially supported in the low-frequency regime as  $\delta\rightarrow0$.  Therefore, we shall  assume that there exists some $K=K(\delta)\in\mathbb{N}$ such that 
\begin{align} \label{low-freq}
  K\Omega\rightarrow0 \quad  \mbox{and} \quad  \sum\limits_{n=-\infty}^{\infty}||v_n||_{2}=\sum\limits_{n=-K}^{K}||v_n||_{2}+o(1),
\end{align}
as $\delta\rightarrow0$, where the sequence of functions $(v_n)_{n \in \mathbb{Z}}$ is a nontrivial solution to (\ref{eq:1DL_system}).\par 
In order to perform some numerical and analytic analysis in this regime, we adapt the Dirichlet-to-Neumann approach of \cite{feppon_cheng:hal-03697696,feppon:hal-03659025} to the one-dimensional, quasi-periodic and time-modulated case to solve (\ref{eq:1DL_system}).

\section{Exact solution}\label{sec:chpt3}
In this section we seek to solve the coupled Helmholtz problem (\ref{eq:1DL_system}) exactly. We first present a characterization of the solution to the exterior problem and then to the interior problem. Lastly, we use the Dirichlet-to-Neumann map to derive a system of equations based on the boundary condition.
\subsection{Exterior problem}
In this section we seek to characterize the Dirichlet-to-Neumann map of the Helmholtz operator on the domain $(0,L)$ with the quasi-periodic boundary condition.\par
We denote the Sobolev space of quasi-periodic complex-valued functions by $H^{1}_{\mathrm{per},\alpha}(\mathbb{R})$. We also denote by $\mathbb{C}^{2N,\alpha}$ the set of quasi-periodic boundary data $f\equiv (f_i^{\pm})_{i\in\mathbb{Z}}$ such that
\begin{align}
    f_{i+N}^{\pm}=\mathrm{e}^{\mathrm{i} \alpha L} f_i^{\pm},
\end{align} 
where  $f_i^{+}$ (resp. $f_i^{-}$) refers to the component associated with $x_i^{+}$ (resp. with
$x_i^{-}$). The space of such quasi-periodic sequences is clearly finite-dimensional; specifically, it is of dimension $2N$.\par 
The following lemma from \cite{feppon_cheng:hal-03697696} provides an explicit expression for the solution to the exterior problem on $\mathbb{R}\setminus\left(D+L\mathbb{Z}\right)$.
\begin{lemma}\label{lemma:DTN}
    Assume that $k^n=(\omega+n\Omega)/v_0$, for some fixed $n\in\mathbb{Z}$, is not of the form $m\pi/\ell_{i(i+1)}$ for some non-zero integer $m\in\mathbb{Z}\backslash\{0\}$ and index $1\leq i\leq N$. Then, for any quasi-periodic sequence $(f_i^{\pm})_{1\leq i\leq N}\in\mathbb{C}^{2N,\alpha}$, there exists a unique solution  $v_{f,n}^{\alpha}\in H^{1}_{\mathrm{per},\alpha}(\mathbb{R})$ to the exterior problem
    \begin{equation}\label{eq:defDTN}
    \left\{\begin{aligned}
        &\left(\frac{\mathrm{d}^2}{\mathrm{d} x^2}+(k^n)^{2}\right)v_{f,n}^{\alpha}=0 &\text{in } \mathbb{R} \backslash (D+L\mathbb{Z}),\\
        &v_{f,n}^{\alpha}(x_i^\pm) = f_i^{\pm} &\text{for all }  1\leq i \leq N,\\
        &v_{f,n}^{\alpha}(x+L) = \mathrm{e}^{\mathrm{i} \alpha L}v_{f,n}^{\alpha}(x)  &\text{in } \mathbb{R} \setminus (D+L\mathbb{Z}).
    \end{aligned}\right.
    \end{equation}
    Furthermore, when $k^n\neq 0$, the solution $v_{f,n}^{\alpha}$ reads explicitly 
    \begin{equation}\label{eq:3dqrq}
        v_{f,n}^{\alpha}(x) =   \alpha_{i}^n \mathrm{e}^{\mathrm{i} k^n x}+\beta_{i}^n \mathrm{e}^{-\mathrm{i} k^nx}  \text{ if }x\in
            (x_i^{+},x_{i+1}^{-}),\qquad \forall i\in\mathbb{Z},
\end{equation}
    where, for fixed $n\in\mathbb{Z}$, $\alpha_i^n$ and $\beta_i^n$ are given by the matrix-vector product
   \begin{equation}
   \label{eqn:pxggr}
        \begin{bmatrix}
           \alpha_i^n\\
           \beta_i^n
        \end{bmatrix} = -\frac{1}{2 \mathrm{i} \sin(k^n \ell_{i(i+1)})} \begin{bmatrix}
            \mathrm{e}^{-\mathrm{i} k^n x_{i+1}^{-}} & -\mathrm{e}^{-\mathrm{i} k^n x_i^{+}} \\
            -\mathrm{e}^{\mathrm{i} k^n x_{i+1}^{-}} & \mathrm{e}^{\mathrm{i} k^n x_i^{+}} 
        \end{bmatrix}
        \begin{bmatrix}
           f_i^{+}\\
           f_{i+1}^{-}
        \end{bmatrix}.
   \end{equation}
\end{lemma}
\begin{proof}
    Identical to \cite[Lemma 2.1]{feppon_cheng:hal-03697696}.
\end{proof}
\begin{definition}
    For any $k^n\in\mathbb{C}$, for fixed $n\in\mathbb{Z}$, which is not of the form $m\pi/\ell_{i(i+1)}$ for some $m\in\mathbb{Z}\backslash\{0\}$ and $1\leq i\leq N-1$, the \textit{Dirichlet-to-Neumann map} with wave number $k^n$ is the linear operator $\mathcal{T}^{k^n,\alpha}\,:\,\mathbb{C}^{2N,\alpha}\to \mathbb{C}^{2N,\alpha}$ defined by 
    \begin{equation}\label{eqn:rm93y}
        \mathcal{T}^{k^n,\alpha}[(f_i^{\pm})_{1 \leq i \leq N}]:=\left(\pm\frac{\mathrm{d} v_{f,n}^{\alpha}}{\mathrm{d} x}(x_i^\pm)\right)_{1 \leq i \leq N},
    \end{equation}
where $v_{f,n}^{\alpha}$ is  the unique solution to (\ref{eq:defDTN}).
\end{definition}
Using the exponential Ansatz presented in Lemma \ref{lemma:DTN} to solve (\ref{eq:defDTN}) gives rise to a closed form definition of the Dirichlet-to-Neumann map, which we introduce in the following proposition.
\begin{prop}\label{prop:DTN}
    For fixed $n\in\mathbb{Z}$, the Dirichlet-to-Neumann map $\mathcal{T}^{k^n,\alpha}$ admits the following explicit matrix representation: for any $k^n\in\mathbb{C}\backslash\{ m\pi/\ell_{i(i+1)}\,:\, m\in\mathbb{Z}\backslash\{0\},\, 1\leq i\leq N-1 \}$, $f\equiv (f_i^{\pm})_{1\leq i\leq N}$, $\mathcal{T}^{k^n,\alpha}[f]\equiv(\mathcal{T}^{k^n,\alpha}[f]_{i}^{\pm})_{1\leq i\leq N}$ is given by
    \begin{equation}\label{eqn:DTNexplicit}
        \begin{bmatrix}
            \mathcal{T}^{k^n,\alpha}[f]_1^{-} \\
            \mathcal{T}^{k^n,\alpha}[f]_1^{+} \\
           \vdots\\
            \mathcal{T}^{k^n,\alpha}[f]_N^{-} \\
            \mathcal{T}^{k^n,\alpha}[f]_N^{+} 
        \end{bmatrix} = \begin{bmatrix}
            -\frac{k^n\cos(k^n\ell_{N(N+1)})}{\sin(k^n\ell_{N(N+1)})} &&&&&
            \frac{k^n}{\sin(k^n\ell_{N(N+1)})}\mathrm{e}^{-\mathrm{i} \alpha L} \\
            & A^{k^n}(\ell_{12}) & & & & \\
            & & A^{k^n}(\ell_{23}) & & & \\
            &  & &\ddots & & \\
            & & & &  A^{k^n}(\ell_{(N-1)N}) &\\
            \frac{k^n}{\sin(k^n\ell_{N(N+1)})}\mathrm{e}^{\mathrm{i} \alpha L}  & & & & & -\frac{k^n\cos(k^n\ell_{N(N+1)})}{\sin(k^n\ell_{N(N+1)})} \\
        \end{bmatrix}  \begin{bmatrix}
            f_1^{-}\\
           f_1^{+}\\
           \vdots\\
           f_N^{-}\\
           f_N^{+}
        \end{bmatrix},
    \end{equation}
    where for any $\ell\in\mathbb{R}$, $A^{k^n}(\ell)$ denotes the $2\times 2$ symmetric matrix 
    \begin{equation}\label{eqn:1lzi8}
        A^{k^n}(\ell):=\begin{bmatrix}
            -\frac{k^n \cos(k^n\ell)}{\sin(k^n\ell)} & \frac{k^n}{\sin(k^n\ell)} \\
            \frac{k^n}{\sin(k^n\ell)} & -\frac{k^n\cos(k^n\ell)}{\sin(k^n\ell)}
        \end{bmatrix}.
    \end{equation}
\end{prop}
\begin{proof}
    Identical to \cite[Proposition 3.3]{jinghao-silvio2023}. 
\end{proof}
So far, we have found a way to solve the exterior problem explicitly for some given boundary data, as stated in Lemma \ref{lemma:DTN}. Moreover, we have provided an explicit matrix representation of the Dirichlet-to-Neumann map in Proposition \ref{prop:DTN}, which we will make use of when dealing with the Neumann boundary condition of (\ref{eq:1DL_system}) in order to solve the interior problem.

\subsection{Interior problem}
Having dealt with the exterior problem in the previous section, we now focus on the solution of the interior problem. We can formulate the interior part of problem (\ref{eq:1DL_system}) using the Dirichlet-to-Neumann map, which leads to
\begin{equation}\label{eq:InteriorProb}
    \left\{\begin{aligned} 
        &\frac{\mathrm{d}^2}{\mathrm{d}x^2}v_{i,n}^*+\frac{\rho_{\mathrm{r}}(\omega+n\Omega)^2}{\kappa_{\mathrm{r}}}v_{i,n}^{**}=0&\text{ in }D+L\mathbb{Z},\\
        &\pm\frac{\mathrm{d}}{\mathrm{d}x}v_{i,n}^*(x_i^{\pm},\alpha)=\delta\mathcal{T}^{k^n,\alpha}[v_n]^{\pm}_i &\text{ for all }i\in\mathbb{Z},\\
        &v_n(x+L,\alpha)=\mathrm{e}^{\mathrm{i}\alpha L}  v_n(x,\alpha)&\text{ for almost every }x\in D+L\mathbb{Z},
    \end{aligned}\right.
\end{equation}
for $n \in \mathbb{Z}$. 
Recall that $v_{i,n}^*$ and $v_{i,n}^{**}$ are the convolutions defined by (\ref{def:conv_v}).\par 
We now recall the definition of a subwavelength quasifrequency \cite{JCP_AMMARI_HILTUNEN}. 
\begin{definition} \label{def:quasifrequency}
    Any frequency $\omega^{\alpha}(\delta) \in [-\Omega/2,\Omega/2)$ for which 
     the $v_n$'s satisfying (\ref{eq:InteriorProb}) are not all trivial  and the corresponding 
     \begin{align}
         u^\alpha(x,t)= \mathrm{e}^{\mathrm{i} \omega^{\alpha}(\delta) t}\sum_{n=-\infty}^\infty v_n(x,\alpha) \mathrm{e}^{\mathrm{i} n \Omega t}
     \end{align}
    is essentially supported  in the low-frequency regime, \textit{i.e.}, there exists $K$ such that (\ref{low-freq}) holds, 
    is called a \textit{subwavelength quasifrequency}. Moreover, $u^\alpha$ is called a \textit{subwavelength Bloch mode} associated to $\omega^{\alpha}(\delta)$.
\end{definition}
Next, we state the following lemma, which provides us with the solution to the interior problem upon using an exponential Ansatz. 
\begin{lemma}
    For each resonator $D_i$, for $i=1,\dots,N$, the interior problem (\ref{eq:InteriorProb}) can be written as an infinitely-dimensional system of ODEs, $\Delta A_i\mathbf{v}_i+B_i\mathbf{v}_i=\mathbf{0}$, with
    \begin{align}
        \mathbf{v}_i:=\begin{bmatrix}
            \vdots \\ v_{i,K} \\ \vdots \\ v_{i,0} \\ \vdots \\ v_{i,-K} \\ \vdots 
        \end{bmatrix}, \,A_i:=\begin{bmatrix}
            & & & & & &\\
            \ddots & & \ddots & & \ddots & & \\
            & r_{i,-M} & \cdots & r_{i,0} & \cdots & r_{i,M} & & \\
            & & \ddots & & \ddots & & \ddots & \\
            & & & & & 
        \end{bmatrix},\nonumber\\ B_i:=\begin{bmatrix}
            & & & & & & &\\
            \ddots & & \ddots & & \ddots & & &\\
            & \gamma^K_{i,-M}(\omega) & \cdots & \gamma^K_{i,0}(\omega) & \cdots & \gamma^K_{i,M}(\omega) & & \\
            & \ddots & & \ddots & & \ddots & & & \\
            & & \gamma^0_{i,-M}(\omega) & \cdots & \gamma^0_{i,0}(\omega) & \cdots & \gamma^0_{i,M}(\omega) &  \\
            & & \ddots & & \ddots & & \ddots & \\
            & & & \gamma^{-K}_{i,-M}(\omega) & \cdots & \gamma^{-K}_{i,0}(\omega) & \cdots & \gamma^{-K}_{i,M}(\omega) \\
            & & & \ddots & & \ddots & & \ddots \\
            & & & & & & &
        \end{bmatrix},
    \end{align}
    where we define the coefficients
    \begin{align}
        \gamma_{i,m}^n(\omega):=\frac{\omega+(n-m)\Omega}{\omega+n\Omega}k_{i,m}\left(k_{\mathrm{r}}^n\right)^2,\quad \forall\,-M\leq m\leq M,\,-\infty<n<\infty,
    \end{align}
    and $v_{i,j}=\left.v_j\right|_{D_i}$. By the definition of $A_i$, the matrix is invertible and we may write $\Delta\mathbf{v}_i+C_i\mathbf{v}_i=\mathbf{0}$, where $C_i:=A_i^{-1}B_i$. Let $\{\Tilde{\lambda}_j^i\}_{j\in\mathbb{Z}}$ be the set of all eigenvalues of $C_i$ with corresponding eigenvectors $\{\mathbf{f}^{j,i}\}_{j\in\mathbb{Z}}$. Using the square-roots $\pm\lambda_j^i$ of the eigenvalues $\Tilde{\lambda}_j^i$, the solution to the interior problem (\ref{eq:InteriorProb}) over $D_i$ takes the form
    \begin{align}\label{eq:vi_ansatz}
        \mathbf{v}_i=\sum\limits_{j=-\infty}^{\infty}\left(a_j^i\mathrm{e}^{\mathrm{i}\lambda_j^ix}+b_j^i\mathrm{e}^{-\mathrm{i}\lambda_j^ix}\right)\mathbf{f}^{j,i},\quad \forall\,x\in\left(x_i^-,x_i^+\right),
    \end{align}
    for coefficients $\{(a_j^i,b_j^i)\}_{j\in\mathbb{Z}}\subset\mathbb{R}^2$.
\end{lemma}
\begin{rem}
    Note that in order to use the Ansatz (\ref{eq:vi_ansatz}) to obtain a solution numerically, we introduce a truncation parameter $K\in\mathbb{N}$ such that $K\geq M$ and such that (\ref{low-freq}) is satisfied. Thus, we truncate the infinitely-dimensional vector and matrices to $\mathbf{v}_i\in\mathbb{R}^{2K+1},\,A_i\in\mathbb{R}^{(2K+1)\times(2K+1)},\,B_i\in\mathbb{R}^{(2K+1)\times(2K+1)}$. Hence, we use the following Ansatz as a solution to the interior problem:
    \begin{align}\label{eq:vi_ansatz_K}
        \mathbf{v}_i=\sum\limits_{j=-K}^{K}\left(a_j^i\mathrm{e}^{\mathrm{i}\lambda_j^ix}+b_j^i\mathrm{e}^{-\mathrm{i}\lambda_j^ix}\right)\mathbf{f}^{j,i},\quad \forall\,x\in\left(x_i^-,x_i^+\right),
    \end{align}
    for coefficients $\{(a_j^i,b_j^i)\}_{-K\leq j\leq K}\subset\mathbb{R}^2$. Moreover, we denote the $n$-th entry of $\mathbf{f}^{j,i}$ by $f_n^{j,i}$. We shall illustrate the influence of the parameter $K$ in Figure \ref{fig:muller_ca_time} and Figure \ref{fig:error_K}.
\end{rem}
Using the Ansatz (\ref{eq:vi_ansatz_K}) and the transmission conditions in (\ref{eq:InteriorProb}), we construct a non-linear system of equations in $\omega$, which characterizes the coefficients $\{(a_j^i,b_j^i)\}_{-K\leq j\leq K}\subset\mathbb{R}^2$ for each $i=1,\dots,N$. 
\begin{thm}\label{lemma:solve_an_bn}
    Let $K$ be a fixed and sufficiently large truncation parameter. The subwavelength quasifrequencies  $\omega$ to the wave problem (\ref{eq:InteriorProb}) are approximately satisfying, as $\delta\to0$, the following truncated system of non-linear equations:
    \begin{align}\label{eq:system_GV}
       \sum\limits_{j=-K}^K\left(\mathcal{G}^{n,j}-\delta\mathcal{T}^{k^n,\alpha}\times\mathcal{V}^{n,j}\right)\mathbf{w}_j=\mathbf{0},\quad\forall\, n, \, -K\leq n\leq K,
    \end{align}
    where the unknown vector is 
    \begin{align}
        \mathbf{w}_j:=\begin{bmatrix}
            a_j^i\\b_j^i
        \end{bmatrix}_{1\leq i\leq N}\in\mathbb{C}^{2N}, \quad\forall\, j \, -K\leq j\leq K,
    \end{align}
    and the matrices $\mathcal{G}^{n,j}=\mathcal{G}^{n,j}(\omega)$ and $\mathcal{V}^{n,j}=\mathcal{V}^{n,j}(\omega)$ are given by 
    \begin{align}
        \mathcal{G}^{n,j}:=\mathrm{diag}\left(\sum\limits_{m=-M}^Mr_{i,m}f^{j,i}_{K+1-n+m}\begin{bmatrix}
            -\mathrm{i}\lambda_j^i\mathrm{e}^{\mathrm{i}\lambda_j^ix_i^-} & \mathrm{i}\lambda_j^i\mathrm{e}^{-\mathrm{i}\lambda_j^ix_i^-} \\ \mathrm{i}\lambda_j^i\mathrm{e}^{\mathrm{i}\lambda_j^ix_i^+} & -\mathrm{i}\lambda_j^i\mathrm{e}^{-\mathrm{i}\lambda_j^ix_i^+}
        \end{bmatrix}\right)_{1\leq i\leq N}\in\mathbb{C}^{2N\times2N}, \nonumber\\
        \mathcal{V}^{n,j}:=\mathrm{diag}\left(f^{j,i}_{K+1-n}\begin{bmatrix}
            \mathrm{e}^{\mathrm{i}\lambda_j^ix_i^-} & \mathrm{e}^{-\mathrm{i}\lambda_j^ix_i^-} \\ \mathrm{e}^{\mathrm{i}\lambda_j^ix_i^+} & \mathrm{e}^{-\mathrm{i}\lambda_j^ix_i^+}
        \end{bmatrix}\right)_{1\leq i\leq N}\in\mathbb{C}^{2N\times2N}.
    \end{align}
    Here, we use the convention $f_{K+1-n}^{k,i}=0$, if $|n|>K$. We define the $2N\times2N(2K+1)$ matrix
    \begin{align}
        \mathcal{A}^n(\omega,\delta):=\begin{bmatrix}
            \mathcal{G}^{n,K}-\delta\mathcal{T}^{k^n,\alpha}\times\mathcal{V}^{n,K} & \cdots & \mathcal{G}^{n,0}-\delta\mathcal{T}^{k^n,\alpha}\times\mathcal{V}^{n,0} & \cdots & \mathcal{G}^{n,-K}-\delta\mathcal{T}^{k^n,\alpha}\times\mathcal{V}^{n,-K}
        \end{bmatrix},
    \end{align}
    such that (\ref{eq:system_GV}) can equivalently be expressed as $\mathcal{A}^n(\omega,\delta)\mathbf{w}=\mathbf{0}$, for all $n=-K,\dots,K$, where
    \begin{align}
        \mathbf{w}:=\begin{bmatrix}
            \mathbf{w}_{K} \\ \vdots \\ \mathbf{w}_{-K}
        \end{bmatrix}\in\mathbb{C}^{2N(2K+1)}.
    \end{align}
    To this end, (\ref{eq:system_GV}) can be written as 
    \begin{align}\label{eq:Aw=0}
        \mathcal{A}^*(\omega,\delta)\mathbf{w}=\mathbf{0},\quad\mathrm{for}\quad\mathcal{A}^*(\omega,\delta):=\begin{bmatrix}
            \mathcal{A}^{K}(\omega,\delta) \\ \vdots \\ \mathcal{A}^{0}(\omega,\delta) \\ \vdots \\ \mathcal{A}^{-K}(\omega,\delta)
        \end{bmatrix}\in\mathbb{C}^{2N(2K+1)\times2N(2K+1)}.
    \end{align}
\end{thm}
\begin{proof}
    Using the Ansatz (\ref{eq:vi_ansatz_K}), we express any $v_n$ on the boundary of a resonator $D_i$ as
    \begin{align}
        v_{n}(x_i^{\pm},\alpha)=\sum\limits_{j=-K}^K\left(a_j^i\mathrm{e}^{\mathrm{i}\lambda_j^ix_i^{\pm}}+b_j^i\mathrm{e}^{-\mathrm{i}\lambda_j^ix_i^{\pm}}\right)f^{k,i}_{K+1-n}.
    \end{align}
    Inserting this into the boundary condition we obtain
    $$
    \begin{array}{l}
        \sum\limits_{j=-K}^K\Bigg(\pm\mathrm{i}\sum\limits_{m=-M}^{M}r_{i,m}\lambda_j^i\left(a_j^{i}\mathrm{e}^{\mathrm{i}\lambda_j^ix_{i}^{\pm}}-b_j^{i}\mathrm{e}^{-\mathrm{i}\lambda_j^ix_i^{\pm}}\right)f^{k,i}_{K+1-j+m}\\
        \qquad-\delta\mathcal{T}^{k^{n},\alpha}\Big[\left(a_j^i\mathrm{e}^{\mathrm{i}\lambda_j^ix}+b_i^n\mathrm{e}^{-\mathrm{i}\lambda_j^ix}\right)f^{k,i}_{K+1-j}\Big]^{\pm}_i\Bigg)=0,
    \end{array}
    $$
    which can be written as (\ref{eq:Aw=0}) by evaluating it for each $-K\leq j\leq K$.
\end{proof}
\begin{figure}[H]
    \centering
    \includegraphics[width=0.58\textwidth]{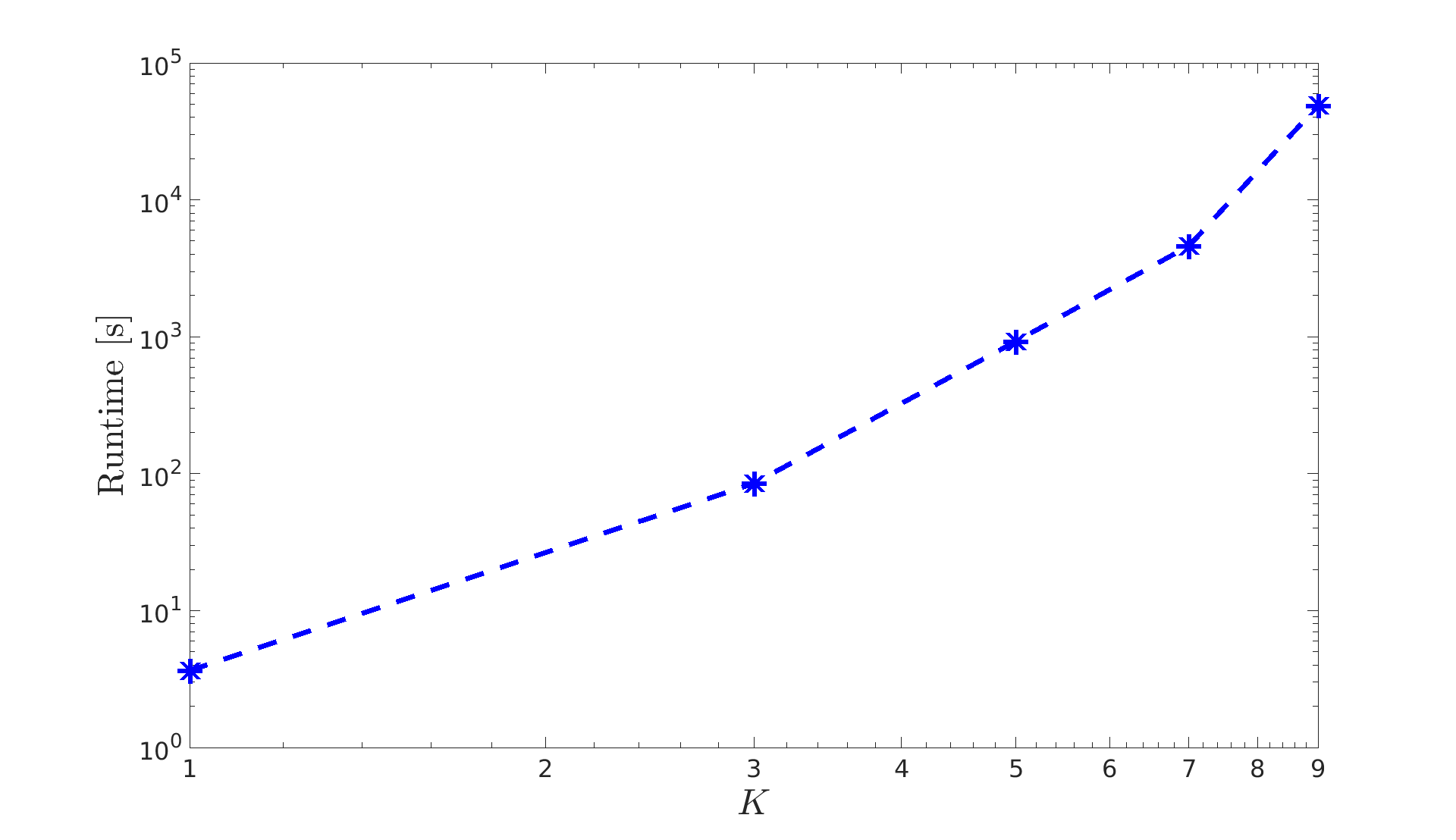}
    \caption{The time it takes to formulate the root-finding problem and for Muller's method to solve it depending on the truncation parameter $K$. The runtime depends algebraically on the parameter $K$.}\label{fig:muller_ca_time}
\end{figure}
The value of $K$ directly impacts the runtime of solving the root-finding problem arising from (\ref{eq:Aw=0}) and the accuracy of the result.  Figure \ref{fig:muller_ca_time} shows how the runtime of Muller's method solving the root-finding problem depends on the truncation parameter $K$. Due to calculation of the eigenvalues of $(2K+1)2N\times(2K+1)2N$ and $(2K+1)\times(2K+1)$ matrices, the runtime increases algebraically in $K$. Contrarily, we require $K$ to be large in order for the Ansatz (\ref{eq:vi_ansatz_K}) to be accurate. Thus, we seek to introduce an alternative way of computing the subwavelength quasifrequencies $\omega_i^{\alpha}$, which does not involve the choice of a truncation parameter $K$. For this matter we shall introduce the formulation of an approximation formula based on the capacitance matrix in \hyperref[sec:cap_approx]{Section 5.1}. 

\subsection{Explicit choice of parameters}\label{sec:rho_kappa}
We now assume that $\rho_i$ and $\kappa_i$ are specifically given by
\begin{align}
    \rho_i(t):=\frac{1}{1+\varepsilon_{\rho,i}\cos\left(\Omega t+\phi_{\rho,i}\right)}, \quad \kappa_i(t):=\frac{1}{1+\varepsilon_{\kappa,i}\cos\left(\Omega t+\phi_{\kappa,i}\right)}, \label{eq:rho_kappa}
\end{align}
for all $1\leq i\leq N$, where  $\varepsilon_{\rho,i},\,\varepsilon_{\kappa,i}$ are the amplitudes of the time-modulations and $\phi_{\rho,i},\,\phi_{\kappa,i}$ the phase shifts.
This means that we can set $M=1$ with the Fourier coefficients defined as follows:
\begin{align}
    &r_{i,-1}:=\frac{\varepsilon_{\rho,i}\mathrm{e}^{-\mathrm{i}\phi_{\rho,i}}}{2},\quad r_{i,0}:=1,\quad r_{i,1}:=\frac{\varepsilon_{\rho,i}\mathrm{e}^{\mathrm{i}\phi_{\rho,i}}}{2},\\
    &k_{i,-1}:=\frac{\varepsilon_{\kappa,i}\mathrm{e}^{-\mathrm{i}\phi_{\kappa,i}}}{2},\quad k_{i,0}:=1,\quad k_{i,1}:=\frac{\varepsilon_{\kappa,i}\mathrm{e}^{\mathrm{i}\phi_{\kappa,i}}}{2}.
\end{align}
The remainder of this paper treats the case of material parameters given by (\ref{eq:rho_kappa}). We will modulate the amplitudes $\varepsilon_{\rho,i}$ and $\varepsilon_{\kappa,i}$ in our numerical experiments to investigate the effect of time-modulated materials on propagating waves.

\section{Muller's method}\label{sec:chpt4}
We solve problem (\ref{eq:1DL_system}) with the help of Muller's method. In particular, we use Lemma \ref{lemma:solve_an_bn} to construct a $2N(2K+1)\times2N(2K+1)$ system of equations $\mathcal{A}^*(\omega,\delta)\mathbf{w}=\mathbf{0}$, which provides the correct coefficients $a_i^n$ and $b_i^n$ of the $n$-th mode $v_n$ in each resonator $D_i$. We seek to find the subwavelength quasifrequencies  $\omega^{\alpha}(\delta)$, which are those values of $\omega$ for which the interior problem (\ref{eq:InteriorProb}) admits a non-trivial solution. Note that these are exactly the values of $\omega$ for which $\mathcal{A}^*(\omega,\delta)$ is non-invertible, \textit{i.e.}, the matrix $\mathcal{A}^*(\omega,\delta)$ has a zero-eigenvalue. Therefore, we define the function
\begin{align}
    f(\omega):=\min\limits_{\lambda\in\sigma\left(\mathcal{A}^*(\omega,\delta)\right)}|\lambda|
\end{align}
whose zeros we must find, for a fixed $\delta$. Note that $\sigma\left(\mathcal{A}^*(\omega,\delta)\right)$ is defined to be the spectrum of the matrix $\mathcal{A}^*(\omega,\delta)$. In order to find the zeros of the non-linear function $f(\omega)$, we use Muller's method upon three initial guesses per root. For a detailed explanation of Muller's method we refer the reader to \cite[Section 1.6]{mcmpp}. One of the reasons for using Muller's method to solve the root-finding problem is that, unlike other root-finding algorithms, Muller's method is well-suited for complex-valued problems \cite[Section 1.6]{mcmpp}.\par 
Muller's method requires the definition of three initial guesses to find a zero of $f(\omega)$. In our numerical computations we make use of the already known definition of the capacitance matrix $C^{\alpha}$ in the static case \cite{jinghao-silvio2023}, as further explained in \hyperref[app:cap]{Appendix A}. Namely, we compute the eigenvalues $\lambda_i^{\alpha},\,1\leq i\leq N$, of the static generalized capacitance matrix $\mathcal{C}^{\alpha}$ and employ the asymptotic approximation from \cite{feppon_cheng:hal-03697696}
\begin{align}\label{eq:approx_omega}
    \omega_i^{\alpha}\approx\pm v_{\mathrm{r}}\sqrt{\lambda_i^{\alpha}\delta},\quad \forall\,1\leq i\leq N.
\end{align}
To initialize Muller's method we use (\ref{eq:approx_omega}) and two perturbations of this value. The use of Muller's method requires the definition of a tolerance, which we chose consistently to be $10^{-12}$.\par 
By the definition of Muller's method we need to supply the algorithm with three initial guesses in order to find a zero of $f(\omega)$, which is not trivial. Furthermore, for fixed $K$, the runtime of Muller's method grows exponentially in $N$, as illustrated in Figure \ref{fig:muller_ca_time}. Therefore, we seek to introduce an alternative characterization of the subwavelength quasifrequencies  $\omega^{\alpha}$, for which we do not require the exact solution of (\ref{eq:InteriorProb}). In view of this, we introduce a discrete approximation of (\ref{eq:1DL_system})  in \hyperref[sec:cap_approx]{Section 5.1}.

\section{Capacitance approximation and asymptotic analysis}\label{sec:chpt5}

\subsection{Capacitance matrix formulation}\label{sec:cap_approx}
By fixing an $\alpha\in Y^*:=(-\pi/L,\pi/L]$, we seek subwavelength quasifrequencies $\omega$ of (\ref{eq:1DL_system}). Following the proof of Lemma 4.1 outlined in \cite{JCP_AMMARI_HILTUNEN}, we can obtain the following result. 
\begin{lemma}\label{lemma:const}
    As $\delta\rightarrow 0$, the functions $v_{i,n}^*(x,\alpha)$ are approximately constant inside the resonator:
    \begin{equation}
        v_{i,n}^*(x,\alpha)\big\rvert_{(x_i^-,x_i^+)} = c_{i,n}+O(\delta^{(1-\gamma)/2}).
    \end{equation}
\end{lemma}
For simplicity of notation, for any smooth function $f:\mathbb{R}\rightarrow \mathbb{R}$, we define $ I_{\partial D_j}[f]$ by
\begin{equation}
    I_{\partial D_j}[f]:= \frac{\mathrm{d} f}{\mathrm{d} x}\bigg\vert_-(x_j^-)- \frac{\mathrm{d} f}{\mathrm{d} x}\bigg\vert_+(x_j^+).
\end{equation}
Passing the convolution in the definition  (\ref{def:conv_v}) of $v_{i,n}^*(x)$ to the time domain, we obtain 
\begin{equation}
    u^*_i(x,t)=\frac{u(x,t)}{\rho_i(t)}=\sum\limits_{n=-\infty}^{\infty}v_{i,n}^*\mathrm{e}^{\mathrm{i}(\omega+n\Omega)t},\quad \forall\,x\in D_i.
\end{equation}
Recall that $1/\rho_i(t)$ has a finite number of Fourier coefficients $\{r_{i,m}\}_{-M\leq m\leq M}$. Thus, $\rho_i(t)$ has an infinite number of Fourier coefficients, which we denote by $\{\Tilde{r}_{i,m}\}_{m\in\mathbb{Z}}$. Therefore, it follows from the definition $u(x,t)=\rho_i(t)u_i^*(x,t)$, for $x\in D_i$, that $v_n(x)$ can be expressed through
\begin{equation}
\label{eq:constvn}
    v_n(x)=\sum\limits_{m=-\infty}^{\infty}\Tilde{r}_{i,m}v^*_{i,n-m}(x),\quad\forall\,x\in D_i.
\end{equation}
\begin{lemma}
    The expression (\ref{eq:constvn}) can be extended to the whole space given by
\begin{equation}
    v_n(x)=\sum\limits_{m=-\infty}^{\infty}\sum\limits_{j=1}^{N}V_i^{\alpha}(x)\Tilde{r}_{j,m}c_{j,n-m} +O(\delta^{(1-\gamma)/2}),
\end{equation}
where the functions $V_i^\alpha:\mathbb{R}\rightarrow \mathbb{R}$ are solutions to the following equations:
	\begin{equation}
		\begin{cases}-\frac{\mathrm{d}^2}{\mathrm{d}x^2} V_i^\alpha=0, & (0,L) \backslash D, \\ V_i^\alpha(x)=\delta_{i j}, & x \in D_j, \\ V_i^\alpha(x+m L)=\mathrm{e}^{\mathrm{i} \alpha m L} V_i^\alpha(x), & m \in \mathbb{Z}.\end{cases}
	\end{equation} 
\end{lemma}
\begin{proof}
    By Lemma \ref{lemma:const}, we have 
    \begin{equation}
        v_n(x)=\sum\limits_{m=-\infty}^{\infty}\Tilde{r}_{i,m}c_{i,n-m}\chi_{D_i} (x)+O(\delta^{(1-\gamma)/2}),\quad\forall\,x\in D.
    \end{equation}
    Between $D_i$ and $D_{i+1}$, $v_n(x)$ satisfies the following Helmholtz equation:
    \begin{equation}
        \left(\frac{\mathrm{d}^2}{\mathrm{d}x^2}+\frac{\rho_0(\omega+n\Omega)^2}{\kappa_0}\right)v_n=0.
    \end{equation}
    The solution is hence given by 
    \begin{equation}
        v_n(x) = A\mathrm{e}^{\mathrm{i}k^nx}+B\mathrm{e}^{-\mathrm{i}k^nx},\quad \forall\,x\in(x_{i}^+,x_{i+1}^-)
    \end{equation}
    for some coefficients $A$ and $B$. We also have $k^n = \sqrt{\frac{\rho_0(\omega+n\Omega)^2}{\kappa_0}}=O(\delta)$. 
    Without loss of generality, we can assume that $x_i^- =0$ and the solution has the following expansion:
    \begin{equation}
        v_n(x) = A+B+\mathrm{i}(A-B)k^nx+O((k^n)^2).
    \end{equation}
    This leads to a linear interpolation between the resonators. Hence, we can extend the characteristic function $\chi_{D_i}$ to $V_i^\alpha$ for all $i=1,...,N$.
\end{proof}
Applying the operator $I_{\partial D_i}$ to $v_n$, we conclude that
\begin{equation}\label{eq:I_vn}
    I_{\partial D_i}[v_n]=\sum\limits_{m=-\infty}^{\infty}\sum\limits_{j=1}^{N}\Tilde{r}_{j,m}c_{j,n-m}C_{ij}^{\alpha},
\end{equation}
where the coefficients of the capacitance matrix $C^{\alpha}$ are given by
\begin{equation}
    C_{ij}^{\alpha}:=I_{\partial D_i}\left[V_j^{\alpha}\right]\quad \mbox{for } i,j=1,\dots, N.
\end{equation}
We give the explicit definition of the coefficients of $C^{\alpha}$ in \hyperref[app:cap]{Appendix A}.\par 
On the other hand, we can use the transmission conditions in (\ref{eq:1DL_system}) to obtain an alternative characterization of $I_{\partial D_i}[v_n]$ as follows:
\begin{equation}\label{eq:I_vn_2}
    \begin{split}
    I_{\partial D_i} [v_n] &= \frac{1}{\delta} \left(\frac{\mathrm{d} v_{i,n}^*}{\mathrm{d} x} \bigg\vert_+(x_i^-,\alpha) - \frac{\mathrm{d} v_{i,n}^*}{\mathrm{d} x} \bigg\vert_-(x_i^+,\alpha) \right)\\
    & = -\frac{1}{\delta} \int_{x_i^-}^{x_i^+} \frac{\mathrm{d}^2v_{i,n}^*}{\mathrm{d} x^2}(x,\alpha)\, \mathrm{d} x \\
    & =  \frac{1}{\delta} \int_{x_i^-}^{x_i^+}  \frac{\rho_{\mathrm{r}}(\omega+n\Omega)^2}{\kappa_{\mathrm{r}}} v_{i,n}^{**}(x,\alpha)  \,\mathrm{d} x.
\end{split} 
\end{equation}
Equating (\ref{eq:I_vn}) and (\ref{eq:I_vn_2}) leads to 
\begin{equation}
    \sum\limits_{m=-\infty}^{\infty}\sum\limits_{j=1}^{N}\Tilde{r}_{j,m}c_{j,n-m}C_{ij}^{\alpha}=\frac{\rho_{\mathrm{r}}\left(\omega+n\Omega\right)^2}{\delta\kappa_{\mathrm{r}}}\int_{x_i^-}^{x_i^+} v_{i,n}^{**}(x,\alpha)\,\mathrm{d}x.
\end{equation}
Next, we define the following functions:
\begin{equation}
    c_i(t)=\mathrm{e}^{\mathrm{i}\omega t}\sum\limits_{n=-\infty}^{\infty}c_{i,n}\mathrm{e}^{\mathrm{i}n\Omega t},\quad V_i(t) = \mathrm{e}^{\mathrm{i}\omega t}\sum\limits_{n=-\infty}^{\infty}\int_{x_i^-}^{x_i^+}v_n(x)\,\mathrm{d}x\,\mathrm{e}^{\mathrm{i}n\Omega t}.
\end{equation}
We have
\begin{equation}
    c_i(t)=\frac{V_i(t)}{\ell_i\rho_i(t)}.
\end{equation}
Letting $\delta\to0$, we obtain
\begin{equation}
    c_i(t)=\mathrm{e}^{\mathrm{i}\omega t}\sum\limits_{n=-M}^{M}c_{i,n}\mathrm{e}^{\mathrm{i}n\Omega t}+o(1),\quad V_i(t)=\mathrm{e}^{\mathrm{i}\omega t}\sum\limits_{n=-M}^{M}\int_{x_i^-}^{x_i^+}v_n(x)\,\mathrm{d}x\,\mathrm{e}^{\mathrm{i}n\Omega t}+o(1).
\end{equation}
To this end, we obtain the following theorem.
\begin{thm}
    Assuming that the material parameters are given by (\ref{material_param}), as $\delta\to0$, the quasifrequencies in the subwavelength regime are, at leading order, given by the quasifrequencies of the system of ordinary differential equations
    \begin{equation}\label{eq:cap_formulation}
        \sum\limits_{j=1}^{N}C_{ij}^{\alpha}w_j(t)=-\frac{\ell_i\rho_{\mathrm{r}}}{\delta\kappa_{\mathrm{r}}}\frac{\mathrm{d}}{\mathrm{d}t}\left(\frac{1}{\kappa_i(t)}\frac{\mathrm{d}w_i(t)}{\mathrm{d}t}\right),\quad\forall\,i=1,\dots,N,
    \end{equation}
    where $w_i(t):=\rho_i(t)c_i(t)$.
\end{thm}
\begin{rem}
    We note here that the capacitance formulation given by (\ref{eq:cap_formulation}) does not depend on the material parameter $\rho_i(t)$. This is a direct consequence of the wave equation (\ref{eq:WaveEq}) and the transmission condition cancelling out the $\rho_i(t)$-dependency. However, (\ref{eq:cap_formulation}) only holds true for small $\delta$ and we may assume that the quasifrequencies for $\delta=O(1)$ do depend on $\rho_i(t)$.
\end{rem}
\begin{rem}
    An equivalent way of formulating the ODE (\ref{eq:cap_formulation}) is through the following system of ODEs:
    \begin{equation}\label{eqn:M_ODE}
        M^{\alpha}(t)\Psi(t)+\Psi''(t)=0,
    \end{equation}
    where $M^{\alpha}(t)=\frac{\delta\kappa_{\mathrm{r}}}{\rho_{\mathrm{r}}}W_1(t)C^{\alpha}W_2(t)+W_3(t)$ with $W_1,W_2$ and $W_3$ being diagonal matrices defined as
    \begin{equation}
        (W_1)_{ii}=\frac{\sqrt{\kappa_i}}{\ell_i},\ \ \ \ (W_2)_{ii}=\sqrt{\kappa_i},\ \ \ \ \ 
        (W_3)_{ii}=\frac{\sqrt{\kappa_i}}{2}\frac{\mathrm{d}}{\mathrm{d} t}\frac{\kappa_i'}{\kappa_i^{3/2}},
    \end{equation}
    for $i=1,\dots,N$, with 
    \begin{equation}
        \Psi(t)=\left(\frac{c_i(t)}{\sqrt{\kappa_i(t)}}\right)_{i=1,\dots, N}.
    \end{equation}
\end{rem}
Our numerical results presented in Figure \ref{fig:muller_vs_ca} corroborate our analytically proven claim that the capacitance matrix approximation is an efficient and effective alternative to the exact computation of the subwavelength quasifrequencies  using Muller's method. 
\begin{figure}[H]
    \begin{subfigure}{0.52\textwidth}
        \centering
        \includegraphics[width=1.0\textwidth]{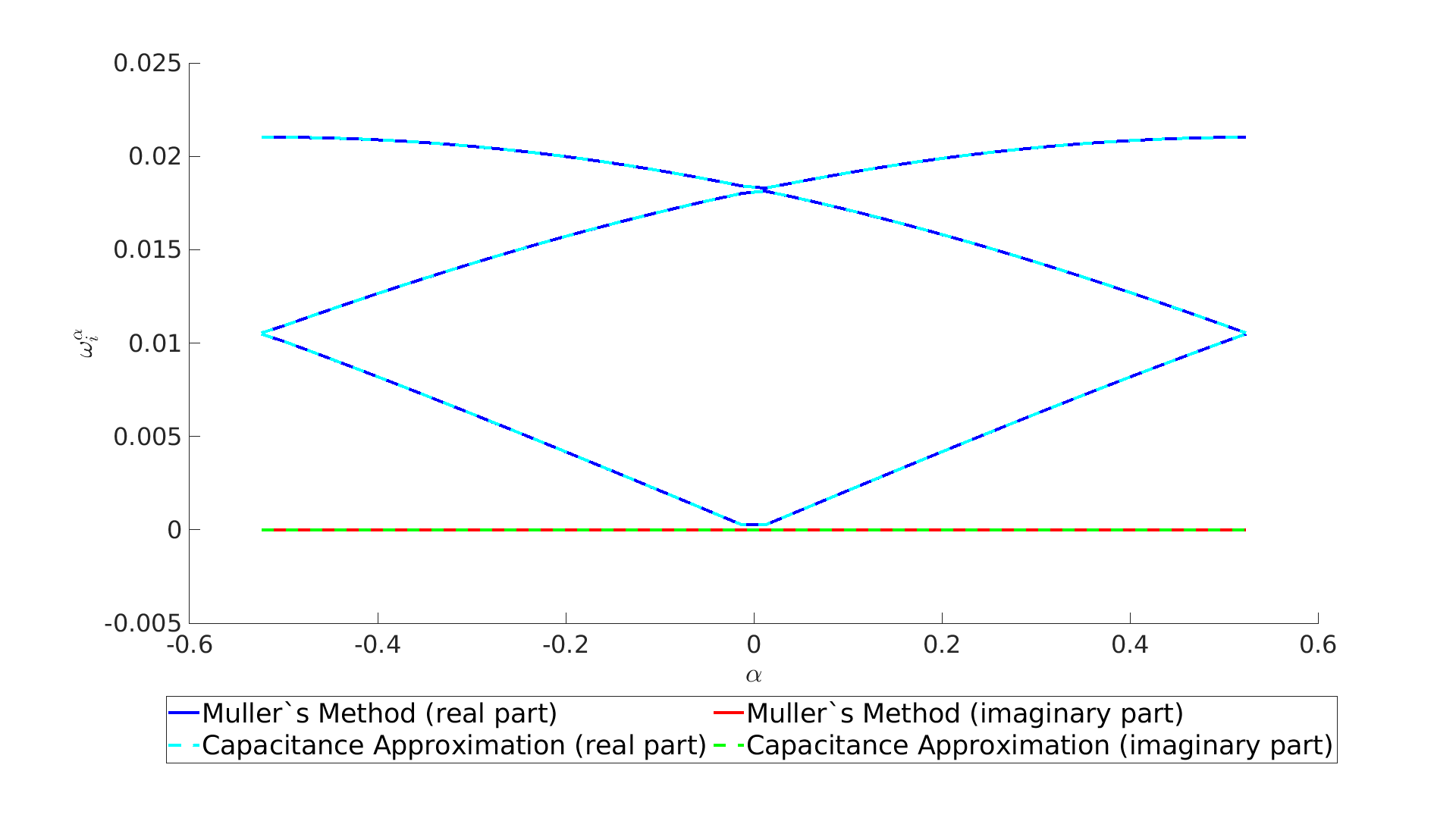}
        \caption{{\footnotesize The results obtained with Muller's method and the capacitance approximation for $\delta=0.0001,\,\Omega=0.05,\,\varepsilon_{\rho}=\varepsilon_{\kappa}=0.4,\,v_0=1,\,v_{\mathrm{r}}=1,\,\phi_{\rho,i}=\phi_{\kappa,i}=\pi/i$, with each resonator being of length $\ell_{i}=1$ with equal spacing $\ell_{ij}=1$. The resulting absolute error is given by $err_{\mathrm{abs}}=1.27\times10^{-6}$.}}
        \label{fig:muller_ca_real_imag}
    \end{subfigure}
    \hspace{0.1cm}
    \begin{subfigure}{0.44\textwidth}
        \centering
        \includegraphics[width=1.0\textwidth]{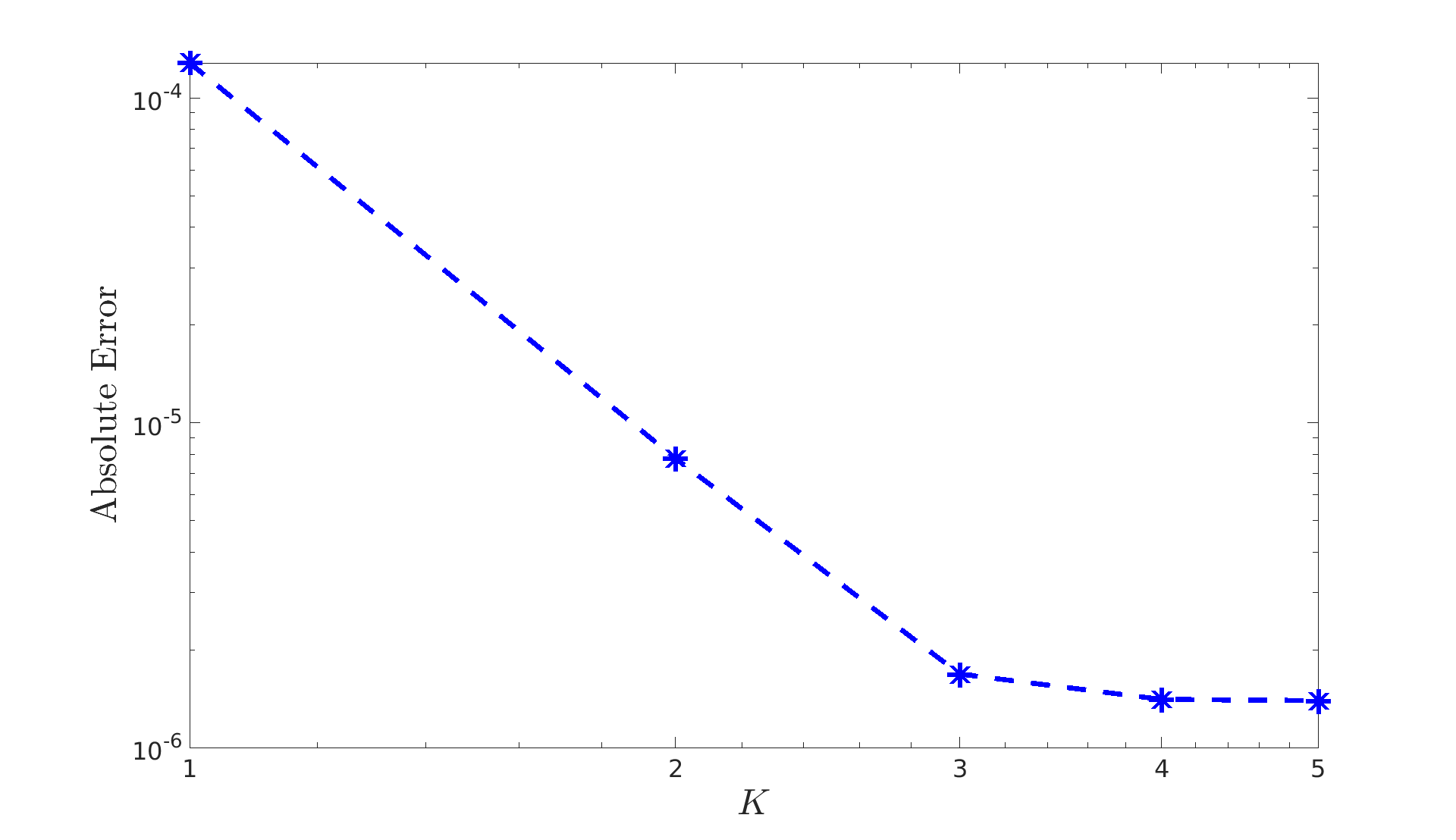}
        \caption{{\footnotesize Our numerical results make it apparent that with increasing $K$, the absolute error decreases. These results were obtained for $N=1,\,\ell_{1}=\ell_{12}=1,\,\delta=0.0001,\Omega=0.05,\,\varepsilon_{\rho}=\varepsilon_{\kappa}=0.4,\,\phi_{\kappa,1}=\phi_{\rho,i}=0,\,v_0=v_{\mathrm{r}}=1$.}}
        \label{fig:error_K}
    \end{subfigure}
    \caption{We compare the quasifrequencies obtained by Muller's method with the quasifrequencies obtained through the capacitance matrix approximation.}
    \label{fig:muller_vs_ca}
\end{figure}
Here, we define the absolute error to be given by
\begin{equation}
    err_{\mathrm{abs}}:=\max_{\alpha\in[-\pi/L,\pi/L]}\,\max_{i=1,\dots,N}\,\bigg|\omega_{i,\mathrm{muller}}^{\alpha}-\omega_{i,\mathrm{cap}}^{\alpha} \bigg|,
\end{equation}
where $\omega^{\alpha}_{i,\mathrm{muller}}$ are the quasifrequencies calculated by Muller's method and $\omega^{\alpha}_{i,\mathrm{cap}}$ are the quasifrequencies calculated through the capacitance approximation.\par 
\begin{figure}[H]
    \centering
    \includegraphics[width=0.58\textwidth]{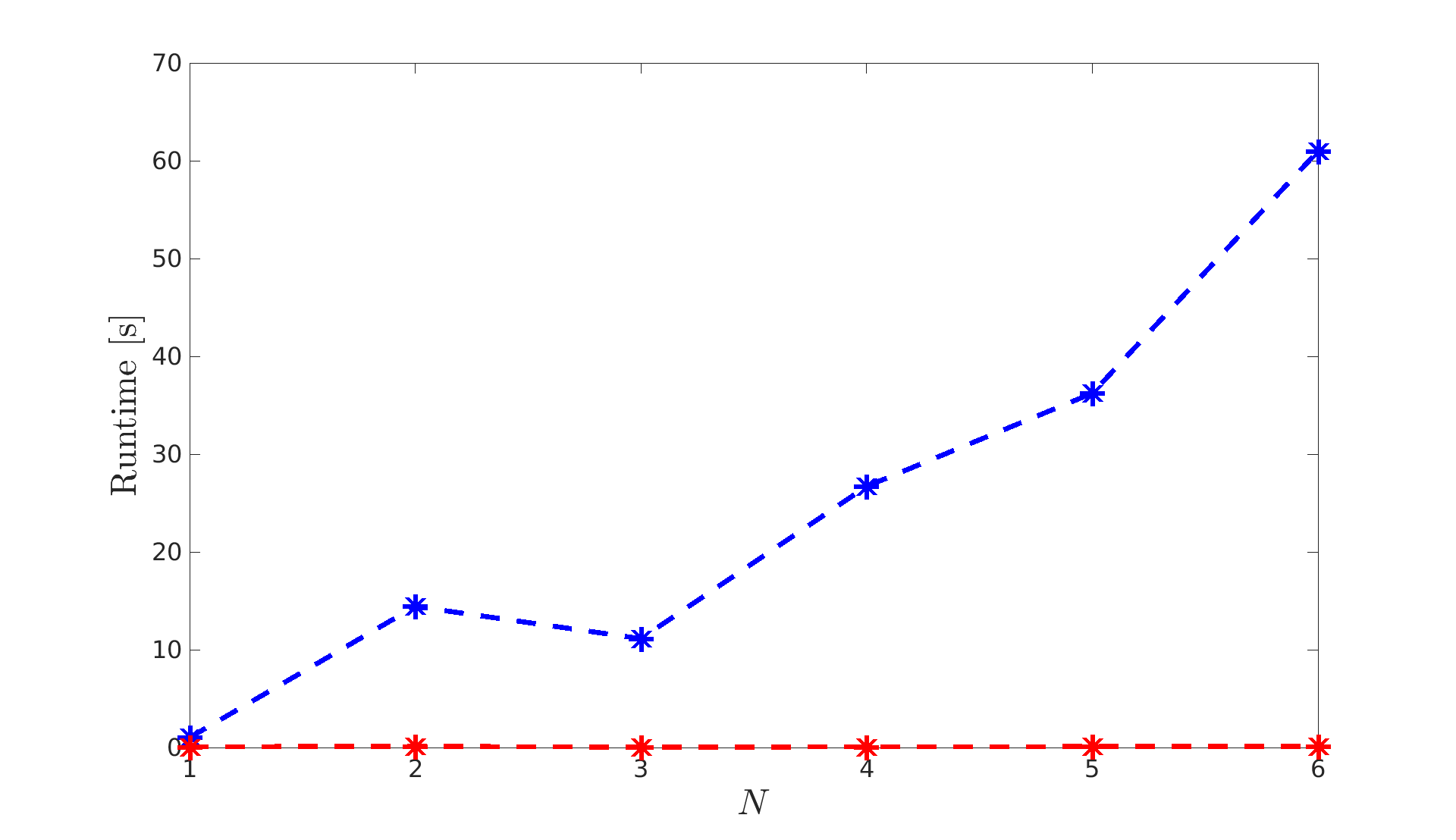}
    \caption{Comparing the runtime of the exact computation for $K=3$ and the capacitance approximation for the following parameter values: $\delta=0.0001,\,\Omega=0.05,\varepsilon_{\kappa}=\varepsilon_{\rho}=0.4,\,\phi_{i}=\pi/i,\,v_0=v_{\mathrm{r}}=1,\,\ell_i=\ell_{i(i+1)}=1$.}
    \label{fig:time_N_muller_cap}
\end{figure}
Based on our results shown in Figure \ref{fig:error_K} we conclude that one should choose $K\geq3$ in order to minimize the absolute error between the result obtained by the capacitance approximation and Muller's method. However, it is important to note that as $K$ increases, the calculation of the eigenvalues of the $(2K+1)2N\times(2K+1)2N$ and $(2K+1)\times(2K+1)$ matrices becomes more complex and thus more error-prone. Therefore, the choice of the parameter $K$ is a delicate matter and we shall replace the exact computation of $\omega_i^{\alpha}$ by the herein introduced capacitance approximation, which is independent of $K$. Nevertheless, the capacitance matrix depends on the number $N$ of resonators. But, the influence of $N$ on the runtime of the capacitance approximation is less significant than the influence it has in the exact computation. We conclude this from Figure \ref{fig:time_N_muller_cap}. This  observation further motivates the introduction of the capacitance matrix approximation.

\subsection{Asymptotic analysis} \label{sec:subasympt}
In order to analyze the reciprocity properties of the wave transmission we make use of asymptotic Floquet analysis developed in \cite{AMMARI_KOSCHE2022227} and we closely follow \cite{ammari_cao_transmprop}. For simplicity, we consider the modulation amplitudes of $\rho_i$ and $\kappa_i$ to be the same over all resonators, \textit{i.e.}, $\varepsilon_{\rho,i}=\varepsilon_{\kappa,i}=\varepsilon$, for all $1\leq i\leq N$. Then we 
assume that the matrix $M^{\alpha}(t)$ in (\ref{eqn:M_ODE}) is  analytic in $\varepsilon$, whence, we can expand $M^{\alpha}(t)$ as follows \cite{ammari_cao_transmprop}:
\begin{align}\label{eqn:M_expansion}
    M^{\alpha}(t)=M_0^{\alpha}+\varepsilon M_1^{\alpha}(t)+\dots+\varepsilon^nM_n^{\alpha}(t)+\dots,
\end{align}
for small $\varepsilon>0$. If $\rho_i(t)$ and $\kappa_i(t)$ have finitely many Fourier coefficients, we can assume that the series (\ref{eqn:M_expansion}) converges for any $|\varepsilon|<\varepsilon_0$, for some $\varepsilon_0>0$ \cite{AMMARI_KOSCHE2022227,Ammari_nonrecip}. Note that we omit the superscript $\alpha$ in the remainder of this section for the sake of convenience. Next, we rewrite the second order ODE (\ref{eqn:M_ODE}) into the first order ODE \cite{ammari_cao_transmprop}
\begin{align}\label{eqn:1st_ord_ODE}
    \frac{\mathrm{d}\mathbf{y}}{\mathrm{d}t}(t)=A(t)\mathbf{y}(t),\quad A(t):=\begin{bmatrix}
        0 & \mathrm{Id}_N \\ -M(t) & 0
    \end{bmatrix},
\end{align}
where $\mathrm{Id}_N $ is the $N\times N$ identity matrix. 
By Floquet's theorem, the fundamental solution of (\ref{eqn:1st_ord_ODE}) can be written as $X(t)=P(t)\mathrm{e}^{Ft}$, for some matrices $P(t)$ and $F$ \cite{Teschl}. As a consequence of $M(t)$ being analytic in $\varepsilon$, we can write \cite{ammari_cao_transmprop}
\begin{align}
    \begin{cases}
        A(t)=A_0+\varepsilon A_1(t)+\dots+\varepsilon^nA_n(t)+\dots,\\
        P(t)=P_0+\varepsilon P_1(t)+\dots+\varepsilon^nP_n(t)+\dots,\\
        F=F_0+\varepsilon F_1+\dots+\varepsilon^nF_n+\dots.
    \end{cases}
\end{align}
The coefficients $A_0$ and $P_0$ are not time-dependent, as they correspond to $\varepsilon=0$, which represents exactly the static case. Due to the $T$-periodicity of the material parameters, $A(t)$ is $T$-periodic and, thus, $A_j(t)$ are also $T$-periodic, for all $j\geq1$. Hence, we may write \cite{ammari_cao_transmprop}
\begin{align}
    A_j(t)=\sum\limits_{m=-\infty}^{\infty}A_j^{(m)}\mathrm{e}^{\mathrm{i}\Omega mt}.
\end{align}
We now aim to derive asymptotic expansions of the eigenvalues $f=f_0+\varepsilon f+\dots$ of $F$ in $\varepsilon$. Assume the first coefficient $A_0$ in the expansion of $A(t)$ to be diagonal. Then, according to \cite{ammari_cao_transmprop}, we have
\begin{align}
    F_0=A_0-\mathrm{i}\Omega\begin{bmatrix}
        m_1 & & \\ & \ddots & \\ & & m_n
    \end{bmatrix}
\end{align}
with $m_i$ being the folding number of $(A_0)_{ii}$, which is defined as follows. 
\begin{definition} \label{deffolding}
    Let $\omega_{A_0}$ be the imaginary part of an eigenvalue of the matrix $A_0$. Then, we can uniquely write $\omega_{A_0}=\omega_0+m\Omega$, where $\omega_0\in[-\Omega/2,\Omega/2)$. The integer $m$ is called the \textit{folding number} \cite{ammari_cao_transmprop}.
\end{definition}
We are specifically interested in investigating perturbations due to the modulations at a degenerate point $f_0$ of $F_0$, which can be obtained through folding, for which we make use of the following lemma from \cite{Ammari_nonrecip}.
\begin{lemma}\label{lemma:ammari_cao}
    The following holds:
    \begin{itemize}
        \item $\left(F_1\right)_{jj}=\left(A_1^{(0)}\right)_{jj}$, for all $j=1,\dots,N$;
        \item For $l\neq j$, we have
        \begin{align}
            (F_1)_{jl}=\begin{cases}
                \left(A_1^{(m_j-m_l)}\right)_{jl}, &\mathrm{if}\,(F_0)_{jj}=(F_0)_{ll},\\
                \left((F_0)_{ll}-(F_0)_{jj}\right)\sum\limits_{m=-\infty}^{\infty}\frac{(A_1^m)_{jl}}{\mathrm{i}\Omega m+(A_0)_{ll}-(A_0)_{jj}}, &\mathrm{otherwise.}
            \end{cases}
        \end{align}
    \end{itemize}
\end{lemma}
\begin{proof}
    The first claim is proved in \cite[Lemma 4.3]{Ammari_nonrecip} and the second one in \cite[Lemma 4.4]{Ammari_nonrecip}.
\end{proof}
As a direct consequence of Lemma \ref{lemma:ammari_cao} we can state the following theorem and corollary.
\begin{thm}\label{thm:ammari_cao}
    Let $f_0$ be a degenerate point of $F$ with multiplicity $r$. Then, $F$ has associated eigenvalues given by $f_0+\varepsilon f_i+O(\varepsilon^2)$, where $f_i$, for $i=1,\dots,r$, are the eigenvalues of the $r\times r$ upper-left block of $F_1$ with entries
    \begin{align}
        (F_1)_{lk}=\left(A_1^{(m_l-m_k)}\right)_{lk},\quad \mbox{for } \,l,k=1,\dots,r,
    \end{align}
    where $m_l$ denotes the folding number of the $l$-th eigenvalue of $A_0$.
\end{thm}
\begin{proof}
    This theorem is proved in \cite[Theorem 4.7]{Ammari_nonrecip}.
\end{proof}
\begin{cor}\label{cor:ammari_cao}
If the degenerate points are of order $r=2$ and $A_1^{(0)}=0$, then the eigenvalues $f$ of $F$ associated with the degenerate point $f_0$ are given by 
\begin{align}
    f_{1,2}=f_0\pm\varepsilon\sqrt{(F_1)_{12}(F_1)_{21}}+O(\varepsilon^2).
\end{align}
\end{cor}
Next, we seek to compute the first-order perturbation of the quasifrequencies. Note that Corollary \ref{cor:ammari_cao}  characterizes the perturbation of the quasifrequencies for which the non-zero Fourier coefficients of $A_1$ are used. Therefore, we need to compute the non-zero Fourier coefficients of $M_1$. As proved in \cite[Theorem~5]{ammari_cao_transmprop}, the following asymptotic expansion of $M$ holds if $\rho_i(t)$ and $\kappa_i(t)$ are modulated as defined in (\ref{eq:rho_kappa}) with $\varepsilon_{\rho,i}=\varepsilon_{\kappa,i}=\varepsilon$, for all $1\leq i\leq N$:
\begin{align*}
    M_{lj}:=\begin{cases}
        L_{lj}+\varepsilon L_{lj}\left(\cos{\left(\Omega t+\phi_{\rho,l}\right)}-\cos{\left(\Omega t+\phi_{\rho,j}\right)}\right.\\
        \left.\qquad-\frac{1}{2}\left(\cos{\left(\Omega t+\phi_{\kappa,l}\right)}+\cos{\left(\Omega t+\phi_{\kappa,j}\right)}\right)\right)+O(\varepsilon^2), &l\neq j,\\
        L_{ll}+\varepsilon\left(\frac{\Omega^2}{2}-L_{ll}\right)\cos{\left(\Omega t+\phi_{\kappa,l}\right)}+O(\varepsilon^2), &l=j.
    \end{cases}
\end{align*}
Note that the quantity $2\varepsilon\sqrt{(F_1)_{12}(F_1)_{21}}$ provides some information about the size of the band gap and in order to compute the coefficients of $F_1$ we need the definition of $M_{lj}$.

\section{Physical interpretation and numerical simulations}\label{sec:chpt6}
The numerical results presented in this section are obtained for $\rho_i(t)$ and $\kappa_i(t)$ defined by (\ref{eq:rho_kappa}), where we vary the parameters $\varepsilon_{\rho,i},\,\varepsilon_{\kappa,i},\,\phi_{\rho,i},\,\phi_{\kappa,i}$. Note that $\varepsilon_{\rho,i}=\varepsilon_{\kappa,i}=0$ corresponds to the static case. In the upcoming notation we omit the subscript $1\leq i\leq N$ of a parameter, if we assume the parameter to be constant over the resonators $D_i,\,1\leq i\leq N$.\par  
Having studied the effect of small periodic perturbations of the material parameters on subwavelength quasifrequencies analytically in \hyperref[sec:subasympt]{Section 5.2}, we now want to validate these results numerically. We use the capacitance matrix approximation in order to conduct some numerical experiments under different conditions. We seek to analyze the so-called \textit{band structure} of the material, which describes the quasifrequency-to-momentum relationship of the propagating waves \cite{ammari_cao_transmprop}. We are especially interested in the occurrence of band gaps and k-gaps as a consequence of time-modulated material parameters. Previous work \cite{AMMARI_FITZPATRICK} has proven the occurrence of subwavelength gaps in three-dimensional high-contrast materials in the static regime. \par
Moreover, we want to understand the effect of time-modulation on the reciprocity of wave transmission properties. The reciprocity of wave transmission is defined as follows.
\begin{definition}
    We say that a wave propagates \textit{reciprocally} if for each $\alpha$ in the space Brillouin zone $Y^*$, the quasifrequencies of the wave problem (\ref{eq:WaveEq}) at $\alpha$ coincide with the quasifrequencies at $-\alpha$ \cite{ammari_cao_transmprop}.
\end{definition}
\begin{figure}[H]
    \begin{subfigure}{0.48\textwidth}
        \centering
        \includegraphics[width=1.0\textwidth]{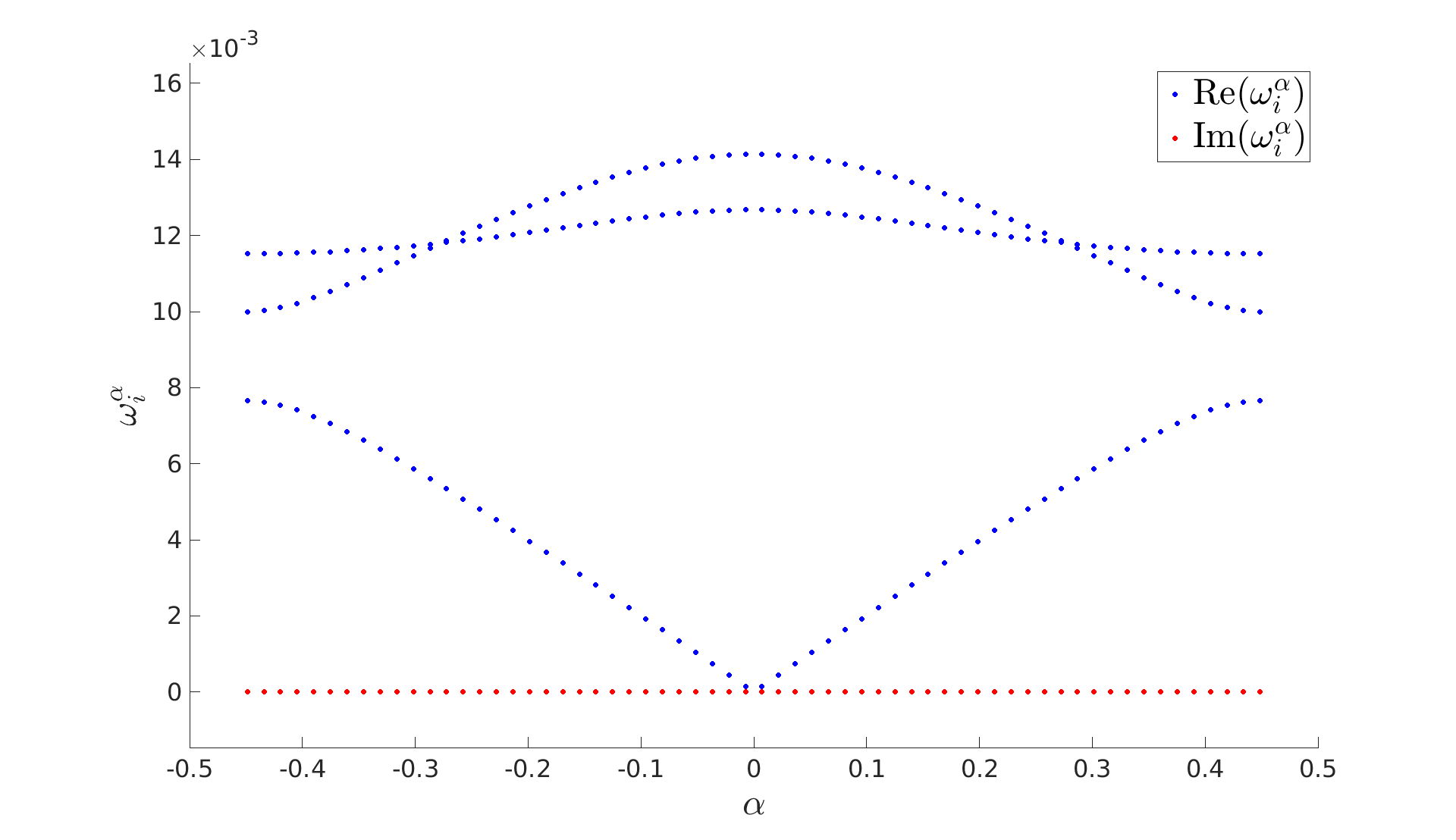}
        \caption{ We consider a setting with $\delta=0.0001,\, \Omega=0.03,\, v=1,\, v_{\mathrm{r}}=1$. We assume the material parameters $\rho$ and $\kappa$ to be static, \textit{i.e.}, $\varepsilon_{\rho}=\varepsilon_{\kappa}=0$.}
        \label{fig:im_re_omega_alpha_static_N3}
    \end{subfigure}
    \hspace{0.1cm}
    \begin{subfigure}{0.48\textwidth}
        \centering
        \vspace{1.5cm}
        \includegraphics[width=1.0\textwidth]{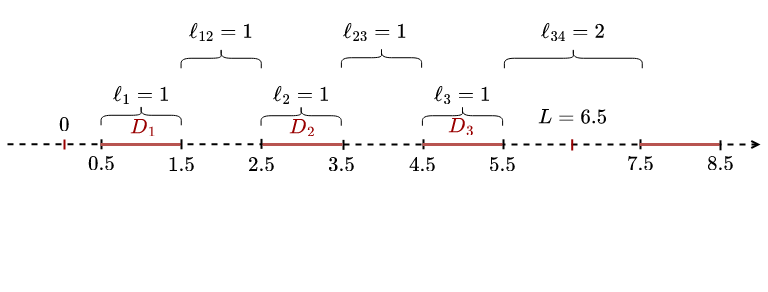}
        \vspace{0.67cm}
        \caption{ Assume that the resonators in the unit cell are each of length $\ell_1=\ell_2=\ell_3=1$ with spacing $\ell_{12}=\ell_{23}=1,\,\ell_{34}=2$. This leads to $L=6.5$.}
        \label{fig:1D_N=3_sketch}
    \end{subfigure}
    \hspace{0.1cm}
    \begin{subfigure}{0.48\textwidth}
        \centering
        \includegraphics[width=1.0\textwidth]{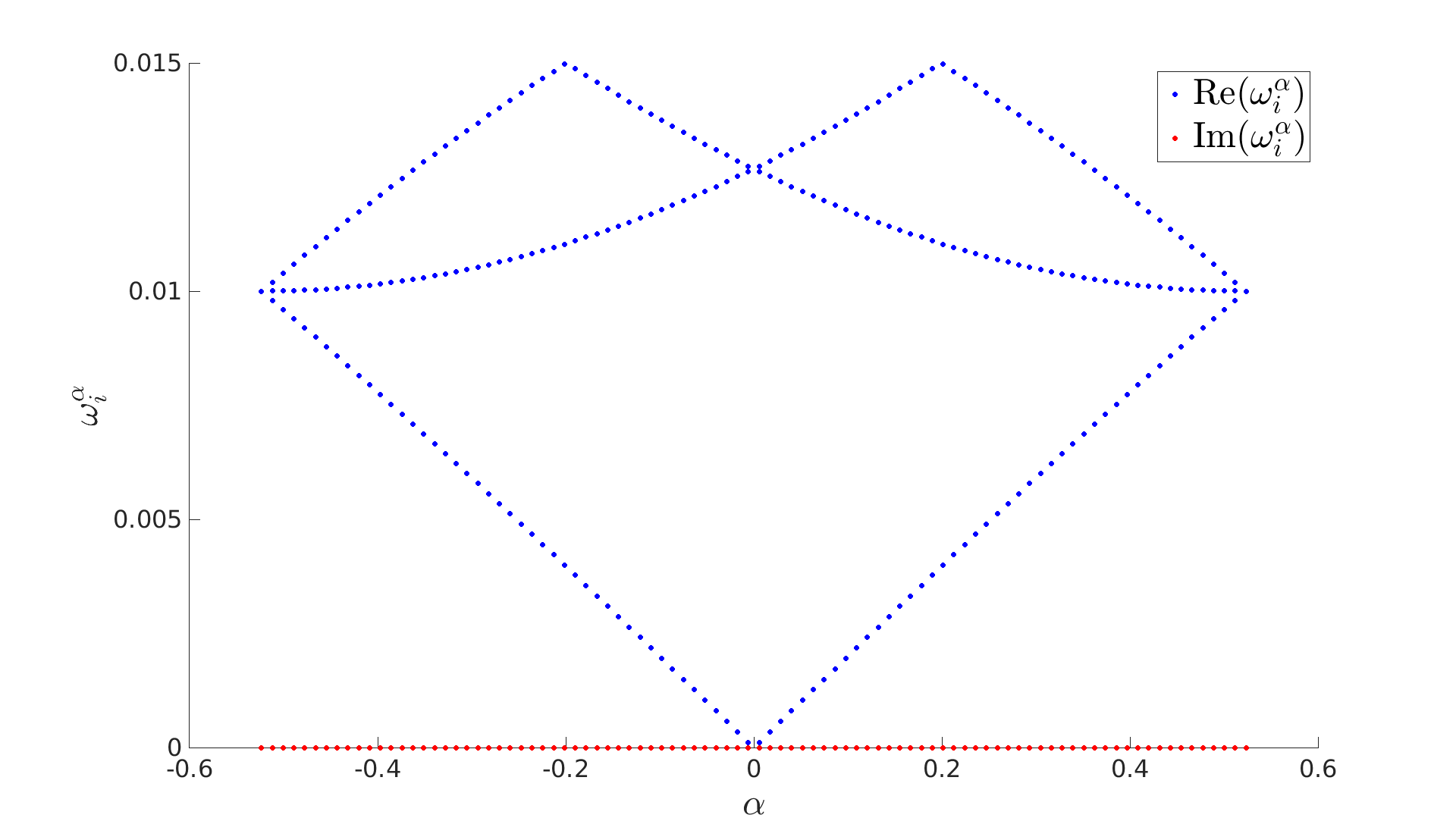}
        \caption{ We consider a setting with $\delta=0.0001,\, \Omega=0.03,\, v=1,\, v_{\mathrm{r}}=1$. We assume the material parameters $\rho$ and $\kappa$ to be static, \textit{i.e.}, $\varepsilon_{\rho}=\varepsilon_{\kappa}=0$.}
        \label{fig:im_re_omega_alpha_static_N3s}
    \end{subfigure}
    \hspace{0.1cm}
    \begin{subfigure}{0.48\textwidth}
        \centering
        \vspace{1.5cm}
        \includegraphics[width=1.0\textwidth]{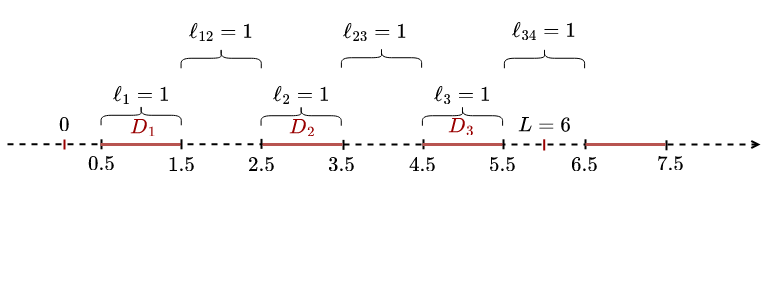}
        \vspace{0.6cm}
        \caption{ Assume that the resonators are each of length $\ell_1=\ell_2=\ell_3=1$ with spacing $\ell_{12}=\ell_{23}=\ell_{34}=1$. This leads to $L=6$.}
        \label{fig:1D_N=3_sketch2}
    \end{subfigure}
    \caption{Subwavelength quasifrequencies  for three resonators repeated periodically in the static case. The figures on the right-hand side illustrate the setting corresponding to the numerical results shown in the left-hand side figures.}
    \label{fig:omega_alpha_N3}
\end{figure}
It becomes apparent from Figure \ref{fig:im_re_omega_alpha_static_N3s} that there is a degenerate point at $\alpha=0$ if the gap size between each resonator is equal, which can be treated equivalently to the case of $N=1$ resonator in the unit cell.\par
Comparing Figure \ref{fig:omega_alpha_N3} with Figure \ref{fig:omega_alpha_N3_mod}, it becomes apparent that modulating $\kappa$ in time turns degenerate points into k-gaps. Furthermore, measuring the size of the k-gaps shows that the gaps forming in the regime $\alpha<0$ do not have the same size as the gaps forming in the regime $\alpha>0$. This means that the wave transmission is non-reciprocal in the time-modulated case. The following theorem has been proven in higher dimensions in \cite[Theorem 8]{ammari_cao_transmprop}, but can equivalently be proven in the one-dimensional case.
\begin{figure}[H]
    \begin{subfigure}{0.48\textwidth}
        \centering
        \includegraphics[width=0.97\textwidth]{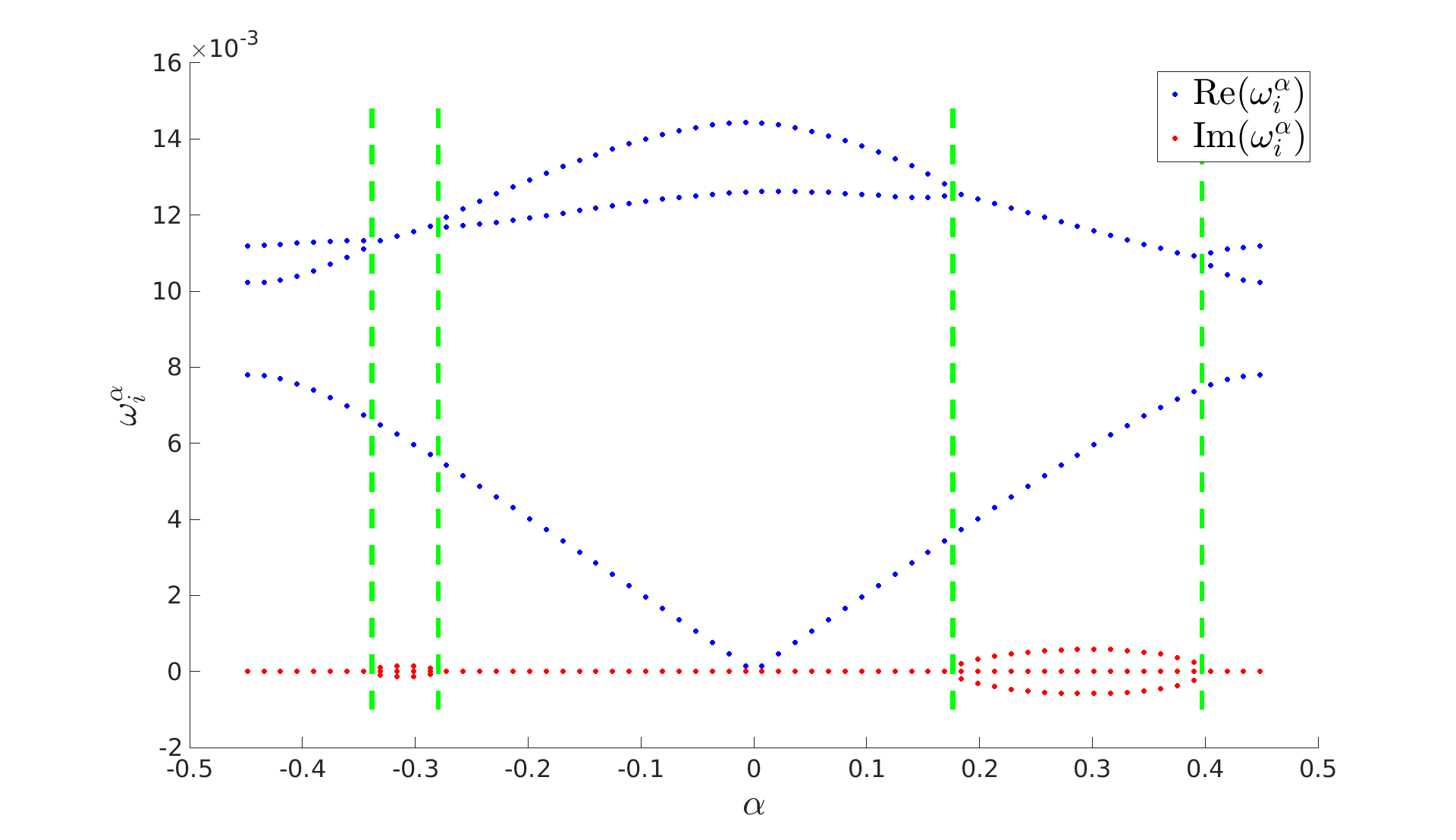}
        \caption{{\footnotesize Assume that $\kappa$ is time-modulated. We consider three resonators repeated periodically each of length $\ell_1=\ell_2=\ell_3=1$ with spacing $\ell_{12}=\ell_{23}=1,\,\ell_{34}=2$.}}
        \label{fig:1D_N=3_kappa}
    \end{subfigure}
    \hspace{0.1cm}
    \begin{subfigure}{0.48\textwidth}
        \centering
        \includegraphics[width=0.97\textwidth]{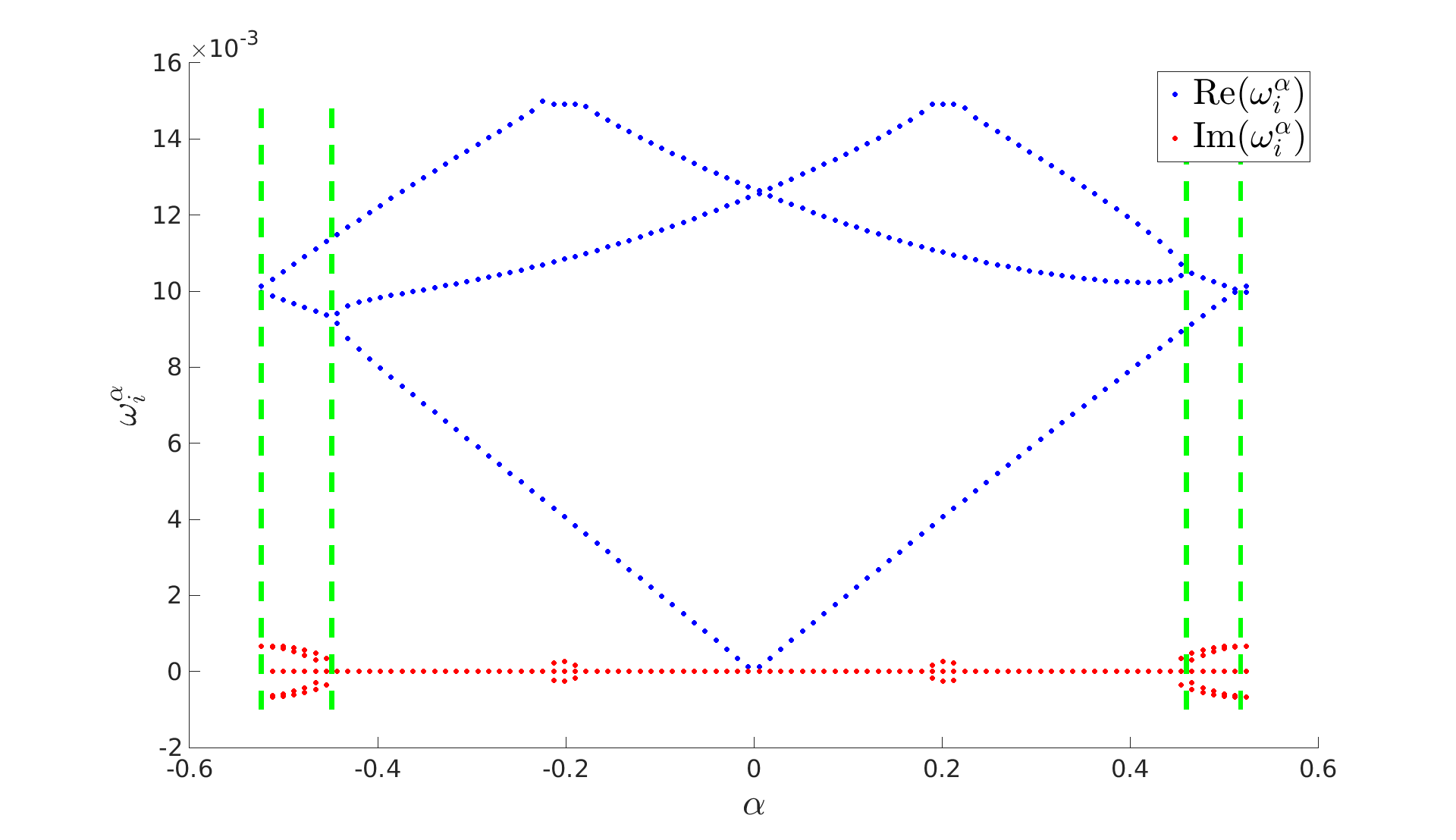}
        \caption{{\footnotesize Assume that $\kappa$ is time-modulated. We consider three resonators repeated periodically each of length $\ell_1=\ell_2=\ell_3=1$ with spacing $\ell_{12}=\ell_{23}=\ell_{34}=1$.}}
        \label{fig:1D_N=3_equidist_kappa}
    \end{subfigure}
    \caption{Subwavelength quasifrequencies  for three resonators repeated periodically in the time-modulated case. We consider a setting with $\delta=0.0001,\,\Omega=0.03,\,v=1,\,v_{\mathrm{r}}=1$. We set the amplitudes for the modulations to be $\varepsilon_{\kappa}=0.2$ with phases $\phi_1=0,\,\phi_2=\pi/2,\,\phi_3=\pi$. The green lines mark the band gaps and k-gaps.}
    \label{fig:omega_alpha_N3_mod}
\end{figure}
\begin{thm}
    If only the material bulk $\kappa$ is time-modulated, then at a degenerate point with multiplicity 2, one of the two Bloch modes is exponentially decaying and the other is exponentially increasing over time. The momentum gaps where waves exhibit this exponential behavior are called the k-gaps.
\end{thm}
\begin{proof}
    Similar to the proof of \cite[Theorem 8]{ammari_cao_transmprop}.
\end{proof}

\section{Conclusion}\label{sec:chpt7}
In this paper we have provided the mathematical foundation to solve the quasi-periodic Helmholtz equation in one dimension with periodically time-dependent material parameters. We presented a discretization of the problem (\ref{eq:1DL_system}), which led to a scheme solving the interior problem exactly up to a negligible numerical error induced by Muller's method. The solution of this scheme involves the calculation of eigenvalues and eigenvectors of large matrices which is very time-consuming.\par 
Furthermore, we have introduced a novel capacitance matrix approximation to the subwavelength quasifrequencies  in one dimension assuming the problem to be quasi-periodic and the material periodically time-modulated, which is equivalent to the approximation formula valid in higher dimensions \cite{JCP_AMMARI_HILTUNEN}. This approximation formula is advantageous because it recovers the quasifrequencies in the subwavelength range much more efficiently. In the static case, the subwavelength quasifrequencies can be approximated by the formula $\omega_i^{\alpha}\approx v_{\mathrm{r}}\sqrt{\lambda_i^{\alpha}\delta}$, for all $1\leq i\leq N$, where $\lambda_i^{\alpha}$ are the eigenvalues of the generalized capacitance matrix $\mathcal{C}^{\alpha}$. Whereas in the time-dependent case, the quasifrequencies $\omega_i^{\alpha}$ are obtained through the solution of an ODE which depends on the capacitance matrix and the material parameters. We have showed in Figure \ref{fig:time_N_muller_cap} that for increasing $N$, approximating the subwavelength quasifrequencies with the help of the capacitance matrix is indeed much faster than computing them with Muller's method. Moreover, the capacitance approximation does not depend on the truncation length $K$ of the system of Helmholtz equations \eqref{eq:1DL_system}, as opposed to the exact solution. We showed numerically in Figure \ref{fig:muller_ca_time} that the runtime of the exact computation of $\omega_i^{\alpha}$ depends heavily on the truncation parameter $K$.\par 
Our numerical analysis led to the conclusion that under time-modulated material parameters, the wave transmission is non-reciprocal, which aligns with the asymptotic analysis in \hyperref[sec:chpt5]{Section 5}. Moreover, it became apparent that periodic time-modulations in the $\kappa_i$'s lead to the formation of k-gaps.

\appendix
\section{Code availability}
The codes that were used to generate the results presented this paper are openly available under \url{https://github.com/liorarueff/1D_quasiperiodic_timemod}. 
\section{Capacitance matrix approximation to the static problem}\label{app:cap}
In this section, we recall results from \cite{jinghao-silvio2023} regarding the capacitance matrix approximation to the static problem.
\begin{definition}
	\label{def:V}
	Consider the solution $V_i^\alpha:\mathbb{R}\rightarrow \mathbb{R}$ of the following problem:
	\begin{equation}
		\begin{cases}-\frac{\mathrm{d}^2}{\mathrm{d}x^2} V_i^\alpha=0, & (0,L) \backslash D, \\ V_i^\alpha(x)=\delta_{i j}, & x \in D_j, \\ V_i^\alpha(x+m L)=\mathrm{e}^{\mathrm{i} \alpha m L} V_i^\alpha(x), & m \in \mathbb{Z}.\end{cases}
	\end{equation}
	The corresponding capacitance matrix is defined by
\begin{align}
	C^\alpha_{ij}&= \frac{\mathrm{d} V_j^\alpha}{\mathrm{d} x}\bigg\vert_-(x_i^-)- \frac{\mathrm{d} V_j^\alpha}{\mathrm{d} x}\bigg\vert_+(x_i^+)\\
	&=-\frac{1}{\ell_{(j-1)j}} \delta_{i(j-1)}+\left(\frac{1}{\ell_{(j-1)j}}+\frac{1}{\ell_{j(j+1)}}\right) \delta_{i j}-\frac{1}{\ell_{j(j+1)}} \delta_{i(j+1)}\nonumber\\
	&\,\quad\quad-\delta_{1 j} \delta_{i N} \frac{\mathrm{e}^{-\mathrm{i} \alpha L}}{\ell_{N(N+1)}}-\delta_{1 i} \delta_{j N} \frac{\mathrm{e}^{\mathrm{i} \alpha L}}{\ell_{N(N+1)}},
\end{align}
or equivalently by
\begin{align}
	C^\alpha=\begin{bmatrix}
		\frac{1}{\ell_{N(N+1)}}+\frac{1}{\ell_{12}} & -\frac{1}{\ell_{12}} & & & -\frac{\mathrm{e}^{-\mathrm{i} \alpha L}}{\ell_{N(N+1)}} \\
		-\frac{1}{\ell_{12}} & \frac{1}{\ell_{12}}+\frac{1}{\ell_{23}} & -\frac{1}{\ell_{23}} & & \\
		& \ddots & \ddots & \ddots & \\
		& & \ddots & \ddots & -\frac{1}{\ell_{{(N-1)N}}} \\
		-\frac{\mathrm{e}^{\mathrm{i} \alpha L}}{\ell_{N(N+1)}} & & & -\frac{1}{\ell_{(N-1)N}} & \frac{1}{\ell_{(N-1)N}}+\frac{1}{\ell_{N(N+1)}}
	\end{bmatrix}.
\end{align}
\end{definition}
The following asymptotics of the band functions hold. 
\begin{prop}
	The first $N$ subwavelength band functions are approximately given by
	\begin{equation}
		\omega_i^\alpha = \sqrt{\delta \lambda_i^{\alpha}}+O(\delta)
	\end{equation}
as $\delta \rightarrow 0$, 	where $\lambda_i^\alpha$ are the eigenvalues of the generalized capacitance matrix $$\mathcal{C}^\alpha:=V^2L^{-1}C^\alpha.$$ Here, $V:=\mathrm{diag}((v_i)_{i=1,\dots,N})$ and $L:=\mathrm{diag}((\ell_i)_{i=1,\dots,N})$.
\end{prop}

\addcontentsline{toc}{chapter}{References}
\renewcommand{\bibname}{References}
\bibliography{refs}

\begin{thebibliography}{10}

\bibitem{jinghao-silvio2023}
Habib Ammari, Silvio Barandun, Jinghao Cao, and Florian Feppon.
\newblock Edge modes in subwavelength resonators in one dimension.
\newblock {\em To appear in Multiscale Model. Simul. (arXiv:2106.12301)}, 2023.

\bibitem{ammari_cao_unidirect}
Habib Ammari and Jinghao Cao.
\newblock Unidirectional edge modes in time-modulated metamaterials.
\newblock {\em Proc. Royal Soc. A}, 478:2022.0395, 2022.

\bibitem{Ammari_nonrecip}
Habib Ammari, Jinghao Cao, and Erik Orvehed~Hiltunen.
\newblock {Non-reciprocal wave propagation in space-time modulated media}.
\newblock {\em Multiscale Model. Simul.}, 20(4):1228--1250, 2022.

\bibitem{ammari_cao_transmprop}
Habib Ammari, Jinghao Cao, and Xinmeng Zeng.
\newblock Transmission properties of space-time modulated metamaterials.
\newblock {\em Stud. Appl. Math.}, 150:558--581, 2023.

\bibitem{ammari.davies.ea2021}
Habib Ammari, Bryn Davies, and Erik~Orvehed Hiltunen.
\newblock Functional analytic methods for discrete approximations of
  subwavelength resonator systems.
\newblock {\em arXiv preprint arXiv:2106.12301}, 2021.

\bibitem{Ammari_Davies}
Habib Ammari, Bryn Davies, and Erik~Orvehed Hiltunen.
\newblock Robust edge modes in dislocated systems of subwavelength resonators.
\newblock {\em J. London Math. Soc.}, 106(3):2075--2135, 2022.

\bibitem{AMMARI202017}
Habib Ammari, Bryn Davies, Erik~Orvehed Hiltunen, and Sanghyeon Yu.
\newblock Topologically protected edge modes in one-dimensional chains of
  subwavelength resonators.
\newblock {\em J. Math. Pures Appl.}, 144:17--49, 2020.

\bibitem{francesco}
Habib Ammari, Francesco Fiorani, and Erik~Orvehed Hiltunen.
\newblock On the validity of the tight-binding method for describing systems of
  subwavelength resonators.
\newblock {\em SIAM J. Appl. Math.}, 82(4):1611--1634, 2022.

\bibitem{Ammari_bubblymedia}
Habib Ammari, Brian Fitzpatrick, David Gontier, Hyundae Lee, and Hai Zhang.
\newblock Minnaert resonances for acoustic waves in bubbly media.
\newblock {\em Ann. Inst. H. Poincar\'e C Anal. Non Lin\'eaire},
  35(7):1975--1998, 2018.

\bibitem{mcmpp}
Habib Ammari, Brian Fitzpatrick, Hyeonbae Kang, Matias Ruiz, Sanghyeon Yu, and
  Hai Zhang.
\newblock {\em Mathematical and Computational Methods in Photonics and
  Phononics}, volume 235 of {\em Mathematical Surveys and Monographs}.
\newblock {American Mathematical Society, Providence, RI}, 2018.

\bibitem{AMMARI_FITZPATRICK}
Habib Ammari, Brian Fitzpatrick, Hyundae Lee, Sanghyeon Yu, and Hai Zhang.
\newblock Subwavelength phononic bandgap opening in bubbly media.
\newblock {\em J. Differential Equations}, 263(9):5610--5629, 2017.

\bibitem{AMMARI_KOSCHE2022227}
Habib Ammari, Erik~O. Hiltunen, and Thea Kosche.
\newblock Asymptotic floquet theory for first order odes with finite fourier
  series perturbation and its applications to floquet metamaterials.
\newblock {\em J. Differential Equations}, 319:227--287, 2022.

\bibitem{JCP_AMMARI_HILTUNEN}
Habib Ammari and Erik~Orvehed Hiltunen.
\newblock Time-dependent high-contrast subwavelength resonators.
\newblock {\em J. Comput. Phys.}, 445:110594, 2021.

\bibitem{ammari2023}
Habib Ammari, Bowen Li, and Jun Zou.
\newblock Mathematical analysis of electromagnetic scattering by dielectric
  nanoparticles with high refractive indices.
\newblock {\em Trans. Amer. Math. Soc.}, 376(1):39--90, 2023.

\bibitem{ammari2017plasmonicscalar}
Habib Ammari, Pierre Millien, Matias Ruiz, and Hai Zhang.
\newblock Mathematical analysis of plasmonic nanoparticles: the scalar case.
\newblock {\em Arch. Ration. Mech. Anal.}, 224(2):597--658, 2017.

\bibitem{hai}
Habib Ammari and Hai Zhang.
\newblock A mathematical theory of super-resolution by using a system of
  sub-wavelength helmholtz resonators.
\newblock {\em Comm. Math. Phys.}, 337(1):379--428, 2015.

\bibitem{hyeonbae}
Kazunori Ando and Hyeonbae Kang.
\newblock Analysis of plasmon resonance on smooth domains using spectral
  properties of the neumann-poincaré operator.
\newblock {\em J. Math. Anal. Appl.}, 435(1):162--178, 2016.

\bibitem{hongyu}
Kazunori Ando, Hyeonbae Kang, and Hongyu Liu.
\newblock Plasmon resonance with finite frequencies: a validation of the
  quasi-static approximation for diametrically small inclusions.
\newblock {\em SIAM J. Appl. Math.}, 76(2):731--749, 2016.

\bibitem{CARMINATI20151}
R.~Carminati, A.~Cazé, D.~Cao, F.~Peragut, V.~Krachmalnicoff, R.~Pierrat, and
  Y.~{De Wilde}.
\newblock Electromagnetic density of states in complex plasmonic systems.
\newblock {\em Surface Science Reports}, 70(1):1--41, 2015.

\bibitem{hongyu3}
Youjun Deng, Hongjie Li, and Hongyu Liu.
\newblock Analysis of surface polariton resonance for nanoparticles in elastic
  system.
\newblock {\em SIAM J. Math. Anal.}, 52(2):1786--1805, 2020.

\bibitem{feppon:hal-03659025}
Florian Feppon and Habib Ammari.
\newblock {Subwavelength resonant acoustic scattering in fast time-modulated
  media}.
\newblock {\em HAL preprint hal-03659025.}, 2022.

\bibitem{feppon_cheng:hal-03697696}
Florian Feppon, Zijian Cheng, and Habib Ammari.
\newblock {Subwavelength resonances in one-dimensional high-contrast acoustic
  media}.
\newblock {\em SIAM J. Appl. Math.}, 83(2):625--665, 2023.

\bibitem{Galiffi}
Emanuele Galiffi, Romain Tirole, Shixiong Yin, Huanan Li, Stefano Vezzoli,
  Paloma~A. Huidobro, M{\'a}rio~G. Silveirinha, Riccardo Sapienza, Andrea
  Al{\`u}, and J.~B. Pendry.
\newblock {Photonics of time-varying media}.
\newblock {\em Advanced Photonics}, 4(1):014002, 2022.

\bibitem{review2}
Hao Ge, Min Yang, Chu Ma, Ming-Hui Lu, Yan-Feng Chen, Nicholas Fang, and Ping
  Sheng.
\newblock Breaking the barriers: advances in acoustic functional materials.
\newblock {\em National Science Review}, 5:159--182, 2018.

\bibitem{haldane2}
F.D.M. Haldane and S.~Raghu.
\newblock Possible realization of directional optical waveguides in photonic
  crystals with broken time-reversal symmetry.
\newblock {\em Phys. Rev. Lett.}, 100(1):013904, 2008.

\bibitem{lemoult_acoustic}
Fabrice Lemoult, Mathias Fink, and Geoffroy Lerosey.
\newblock Acoustic resonators for far-field control of sound on a subwavelength
  scale.
\newblock {\em Phys. Rev. Lett.}, 107:064301, 2011.

\bibitem{lemoult}
Fabrice Lemoult, Nadège Kaina, Mathias Fink, and Geoffroy Lerosey.
\newblock Soda cans metamaterial: A subwavelength-scaled phononic crystal.
\newblock {\em Crystals}, 6(7):82, 2016.

\bibitem{leroy}
Valentin Leroy, Alice Bretagne, Mathias Fink, Hervé Willaime, Patrick
  Tabeling, and Arnaud Tourin.
\newblock Design and characterization of bubble phononic crystals.
\newblock {\em Appl. Phys. Lett.}, 95(17):171904, 2009.

\bibitem{Li_Minnaert}
Hongjie Li, Hongyu Liu, and Jun Zou.
\newblock Minnaert resonances for bubbles in soft elastic materials.
\newblock {\em SIAM J. Appl. Math.}, 82(1):119--141, 2022.

\bibitem{Lin_2022}
Junshan Lin and Hai Zhang.
\newblock Mathematical theory for topological photonic materials in one
  dimension.
\newblock {\em J. Phys. A}, 55(49):495203, 2022.

\bibitem{review}
Guancong Ma and Ping Sheng.
\newblock Acoustic metamaterials: From local resonances to broad horizons.
\newblock {\em Science Advances}, 2(2):e1501595, 2016.

\bibitem{john}
Taoufik Meklachi, Shari Moskow, and John~C. Schotland.
\newblock Asymptotic analysis of resonances of small volume high contrast
  linear and nonlinear scatterers.
\newblock {\em J. Math. Phys.}, 59(8):083502, 2018.

\bibitem{Minnaert_water}
M.~Minnaert.
\newblock On musical air-bubbles and the sounds of running water.
\newblock {\em Phil. Mag.}, 16(104):235--248, 1933.

\bibitem{haldane}
S.~Raghu and F.D.M. Haldane.
\newblock Analogs of quantum-hall-effect edge states in photonic crystals.
\newblock {\em Phys. Rev. A}, 78:033834, 2008.

\bibitem{alu1}
Dimitrios~L. Sounas and Andrea Al{\`u}.
\newblock Non-reciprocal photonics based on time modulation.
\newblock {\em Nature Photonics}, 11:774--783, 2017.

\bibitem{Taravati}
Sajjad Taravati and Ahmed~A. Kishk.
\newblock Space-time modulation: Principles and applications.
\newblock {\em IEEE Microwave Magazine}, 21(4):30--56, 2020.

\bibitem{Teschl}
Gerald Teschl.
\newblock {\em Ordinary Differential Equations and Dynamical Systems}.
\newblock American Mathematical Society, 2012.

\end{thebibliography}
\bibliographystyle{plain}

\end{document}